\documentclass{compositio}
\usepackage{amsmath}
\usepackage{amsthm}
\usepackage{latexsym}
\usepackage{amsfonts}
\usepackage[all]{xy}
\usepackage{amssymb}

\newcommand{\QQ}{\mathbb Q}
\newcommand{\ZZ}{\mathbb Z}
\newcommand{\RR}{\mathbb R}

\newcommand{\BB}{\mathbb B}
\newcommand{\LL}{\mathcal{L}}
\newcommand{\CC}{\mathbb C}

\newcommand{\Qp}{\QQ_p}
\newcommand{\Zp}{\ZZ_p}
\newcommand{\calO}{\mathcal{O}}
\newcommand{\DD}{\mathbb D}
\newcommand{\HH}{\mathcal{H}}
\newcommand{\vp}{\varphi}
\newcommand{\Qpi}{\mathbb{Q}_{p,\infty}}
\newcommand{\Hpm}{\mathbb{H}^1_{\mathrm{Iw},\pm}}
\newcommand{\ff}{\mathfrak{f}}
\newcommand{\K}{\mathcal{K}}
\newcommand{\U}{\mathcal{U}}
\newcommand{\X}{\mathcal{X}}
\newcommand{\C}{\mathcal{C}}
\newcommand{\V}{\mathcal{V}}
\newcommand{\E}{\mathcal{E}}
\newcommand{\Hh}{\mathbb{H}}
\newcommand{\HIw}{\mathbb{H}^1_{\Iw}}
\newcommand{\Qpn}{\mathbb{Q}_{p,n}}
\newcommand{\kato}{{\bf z}^{\rm Kato}}
\newcommand{\OO}{\mathcal{O}}

\DeclareMathOperator{\cris}{cris}
\DeclareMathOperator{\Char}{Char}
\DeclareMathOperator{\cor}{cor}
\DeclareMathOperator{\Sel}{Sel}
\DeclareMathOperator{\Tw}{Tw}
\DeclareMathOperator{\col}{Col}
\DeclareMathOperator{\Iw}{Iw}
\DeclareMathOperator{\Gal}{Gal}

\DeclareMathOperator{\image}{Im}
\DeclareMathOperator{\id}{id}
\DeclareMathOperator{\Fr}{Fr}
\DeclareMathOperator{\Tr}{Tr}

\DeclareMathOperator{\Hom}{Hom}

\newtheorem{theorem}{Theorem}[section]
\newtheorem{proposition}[theorem]{Proposition}
\newtheorem{lemma}[theorem]{Lemma}
\newtheorem{corollary}[theorem]{Corollary}
\newtheorem{remark}[theorem]{Remark}

\newtheorem{conjecture}[theorem]{Conjecture}

\begin{document}
\title{Iwasawa Theory for Modular Forms at Supersingular Primes}
\author{Antonio Lei}

\begin{abstract}
We generalise works of Kobayashi to give a formulation of the Iwasawa main conjecture for modular forms at supersingular primes. In particular, we give analogous definitions of the plus and minus Coleman maps for normalised new forms of arbitrary weights and relate Pollack's $p$-adic $L$-functions to the plus and minus Selmer groups. In addition, by generalising works of Pollack and Rubin on CM elliptic curves, we prove the ``main conjecture" for CM modular forms. 
\end{abstract}
\email{aifl2@cam.ac.uk}
\thanks{The author is supported by Trinity College, Cambridge.}
\address{Department of Pure Mathematics and Mathematical Statistics\\University of Cambridge\\
Wilberforce Road\\Cambridge CB3 0WB\\United Kingdom}
\maketitle


\section{Introduction}

The Taniyama-Shimura conjecture, proved by Wiles et al, asserts that elliptic curves over $\QQ$ correspond to modular forms of weight $2$. Therefore, it is natural to ask which results on elliptic curves can be generalised to modular forms of higher weights. In this paper, we discuss how this can be done for some recent results on supersingular primes.

Let $p$ be an odd prime and let $G_\infty$ be the Galois group of the extension $k_\infty$ of $\QQ$ by $p$ power roots of unity. We denote by $\Lambda(G_\infty)$ the Iwasawa algebra of $G_\infty$ over $\Zp$. If $\Delta$ denotes the torsion subgroup of $G_\infty$ and $\gamma$ is a fixed topological generator of the $\Zp$-part of $G_\infty$, then $\Lambda(G_\infty)\cong\Zp[\Delta][[\gamma-1]]$.

Let $f=\sum a_nq^n$ be a normalised eigen-newform of weight $k \ge 2$, level $N$ and character $\epsilon$. For notational simplicity, we assume $a_p\in\ZZ$ throughout the introduction. We fix $p$ so that $p\nmid N$. Kato \cite{kato04} has formulated a main conjecture relating an Euler system (to which we refer as Kato zeta element) to some cohomological group over $k_\infty$ (see Section~\ref{katoreview} for a brief review).

If $\alpha$ is a root of $X^2-a_pX+\epsilon(p)p^{k-1}$ such that $v_p(\alpha)<k-1$ where $v_p$ is the $p$-adic valuation of $\CC_p$ with $v_p(p)=1$, then there exists a $p$-adic $L$-function $L_{p,\alpha}$ interpolating complex $L$-values of $f$. When $f$ is ordinary at $p$ (i.e. $a_p$ is a $p$-adic unit) and $\alpha$ is the unique unit root of the quadratic above, $L_{p,\alpha}$ lies inside $\QQ\otimes\Lambda(G_\infty)$ and the $p$-Selmer group $\Sel_p(f/k_\infty)$ of $f$ over $k_\infty$ is $\Lambda(G_\infty)$-torsion, i.e. its Pontryagin dual 
\[
\Sel_p(f/k_\infty)^\vee=\Hom_{\rm cts}\left(\Sel_p(f/k_\infty),\Qp/\Zp\right)
\]
is $\Lambda(G_\infty)$-torsion. If $\theta$ is a character on $\Delta$, the $\theta$-isotypical component of $\Sel_p(f/k_\infty)^\vee$ is $\Zp[[\gamma-1]]$-torsion. We can associate to it a characteristic ideal. Kato's main conjecture is equivalent to asserting that this ideal is generated by the $\theta$-component of $L_{p,\alpha}$ (written as $L_{p,\alpha}^\theta$), i.e. there is a pseudo-isomorphism (a homomorphism with finite kernel and cokernel)
\[
\Sel_p(f/k_\infty)^{\vee,\theta}\rightarrow\prod_{i=1}^r\Zp[[\gamma-1]]/(x_i)
\]
for some $x_i\in\Zp[[\gamma-1]]$ such that $x_1\cdots x_r=L_{p,\alpha}^\theta$.

When $f$ is supersingular at $p$ (i.e. $p|a_p$), the $p$-adic $L$-functions of $f$ as given above are not in $\QQ\otimes\Lambda(G_\infty)$ and $\Sel_p(f/k_\infty)$ is not $\Lambda(G_\infty)$-cotorsion (see Section~\ref{nottorsion}). Therefore, Kato's main conjecture cannot be reformulated in the same way as the oridnary case. 

In recent years, much progress has been made on supersingular primes. When $a_p=0$, Pollack \cite{pollack03} has defined the plus and minus $p$-adic $L$-functions $L_p^{\pm}$ which have bounded coefficients. In \cite{kobayashi03}, again assuming $a_p=0$, Kobayashi defined the plus and minus Selmer groups $\Sel_p^\pm$ for the case when $f$ corresponds to an elliptic curve $\E$ over $\QQ$ and proved that $\Sel_p^\pm(\E/k_\infty)$ are $\Lambda(G_\infty)$-cotorsion. It is then possible to reformulate Kato's main conjecture as follows.

\begin{conjecture}\label{main}
Let $\theta$ be a character on $\Delta$. Under the notation above, the characteristic ideal of $\Sel_p^\pm(\E/k_\infty)^{\vee,\theta}$ is generated by $L_p^{\pm,\theta}$.
\end{conjecture}
One inclusion of conjecture~\ref{main}, namely $L_p^{\pm,\theta}$ does lie inside the said characteristic ideal, follows from that of Kato's main conjecture under some assumptions. For the CM case, the other inclusion has been proved by Pollack and Rubin in \cite{pollackrubin04}, using the theory of imaginary quadratic fields and elliptic units.

We now explain how $\Sel_p^\pm(\E/k_\infty)$ is defined. Let $\mu_{p^n}$ be the set of $p^n$th roots of unity. The idea of Kobayashi is to define subgroups $\E^\pm(\Qp(\mu_{p^n}))$ of $\E(\Qp(\mu_{p^n}))$ which can be identified with its image in $H^1(\Qp(\mu_{p^n}),\E[p^\infty])$ under the Kummer map. The $\pm$-Selmer groups over $\QQ(\mu_{p^n})$ is defined to be
\[
\Sel_p^\pm(\E/\QQ(\mu_{p^n}))=\ker\left(\Sel_p(\E/\QQ(\mu_{p^n}))\rightarrow\frac{H^1(\Qp(\mu_{p^n}),\E[p^\infty])}{\E^\pm(\Qp(\mu_{p^n}))\otimes\Qp/\ZZ_p}\right).
\] 
Then, $\Sel_p^\pm(\E/k_\infty)$ is defined to be the direct limit of $\Sel_p^\pm(\E/\QQ(\mu_{p^n}))$.

On the one hand, $\E[p^\infty]$ gives a $p$-adic representation of $\Gal(\bar{\QQ}/\QQ)$ and one can define analogous representations for arbitrary modular forms (see \cite{deligne69} for details). On the other hand, the Kummer image of $\E(\Qp(\mu_{p^n}))$ can be identified with the so-called finite cohomology subgroup $H^1_f$ defined in \cite{blochkato}. Therefore, we can give a definition of $\Sel^\pm(f/k_\infty)$ analogously for any modular forms without much difficulty.  

To show that $\Sel_p^\pm(\E/k_\infty)$ is $\Lambda(G_\infty)$-cotorsion, Kobayashi constructed the $\pm$-Coleman maps
\[
\col^\pm:\lim_{\leftarrow}H^1(\Qp(\mu_{p^n}),T_p(\E))\rightarrow\Lambda(G_\infty)
\]
where $T_p(\E)$ denotes the Tate module of $\E$ at $p$. In particular, $\col^\pm$ send the Kato zeta element from \cite{kato04} to $L_p^\pm$ respectively. By applying the Poitou-Tate exact sequence, he then showed that the Pontryagin dual of $\Sel_p^\pm(\E/k_\infty)$ is killed by $L_p^\pm\ne0$, hence $\Lambda(G_\infty)$-cotorsion. 

We follow this strategy to show that $\Sel_p^\pm(f/k_\infty)$ are $\Lambda(G_\infty)$-cotorsion for $f$ of any weights $k\ge2$. Although the Coleman maps in \cite{kobayashi03} are defined using formal groups, they can in fact be obtained from Perrin-Riou's exponential map defined in \cite{perrinriou94}. We make use of this and observe that there is a divisibility phenomenon, similar to that used in the construction of $L_p^\pm$ in \cite{pollack03}. This enables us to construct analogous $\pm$-Coleman maps for general $f$. Although we do not need any restrictions on $p$ to define them, we assume that $p+1\nmid k-1$ in order to describe their kernels, which are related to the local conditions in the definition of $\Sel_p^\pm$. We then formulate a main conjecture as follows.

\begin{conjecture}\label{main2}
Let $f$ and $\theta$ be as above. There exists $n^\pm\in\ZZ$ such that the characteristic ideal of $\Sel_p^\pm(f/k_\infty)^{\vee,\theta}$ is generated by $p^{n^\pm}L_p^{\pm,\theta}$.
\end{conjecture}

As in the case of elliptic curves, Conjecture~\ref{main2} is equivalent to Kato's main conjecture and one inclusion holds.

It has to be pointed out that we are assuming $a_p=0$ as in \cite{kobayashi03} and \cite{pollack03}. Since $|a_p|\le2p^{(k-1)/2}$ (due to Deligne), $a_p$ is always zero when $p>3$ when $f$ corresponds to an elliptic curve. When $k>2$, the assumption is much stronger, although if $f$ is a CM modular form, $a_p=0$ for any supersingular primes $p$ (see Section~\ref{CM}). More recently, Sprung \cite{sprung09} has generalised works of Kobayashi to the case $a_p\ne0$ for elliptic curves over $\QQ$. It would be desirable to know whether this can be done for modular forms of higher weights as well.

The layout of this paper is as follows. We fix some notation and review some basic properties in Section~\ref{not}. In Section~\ref{ap0}, we first review some of the main results which we need from \cite{perrinriou94} and \cite{kato04}. We then construct the $\pm$-Coleman maps. The kernels of these maps are worked out explicitly in Section~\ref{kernel} and their images are described in Section~\ref{images}. Following \cite{kobayashi03}, we define $\Sel_p^\pm$ in Section~\ref{selmer}. We show that they are $\Lambda(G_\infty)$-cotorsion which enables us to formulate the ``main conjecture" for which one inclusion of the conjecture is shown. Finally, in Section~\ref{CM}, the other inclusion is proved in the case of CM modular forms over $\QQ$, following the strategy of \cite{pollackrubin04}.

\acknowledgements{The author would like to thank Prof Tony Scholl for suggesting the study of this topic and his patient guidance and tremendous help. The author is also indebted to Alex Bartel, Tobias Berger, Michael Fester and Byoung Du Kim for the very helpful discussions. Finally, the author is extremely grateful to the anonymous referees for their very useful comments and suggestions.}


 \section{Background}\label{not}

In this section, we fix some notation which is used throughout the paper. We also state some basic properties of some of the objects which we study.

\subsection{Extensions by $p$ power roots of unity} 
   Throughout this paper, $p$ is an odd prime. If $K$ is a field of characteristic $0$, either local or global, $G_K$ denotes its absolute Galois group, $\chi$ the $p$-cyclotomic character on $G_K$ and $\calO_K$ the ring of integers of $K$. For an integer $n\ge0$, we write $K_n$ for the extension $K(\mu_{p^n})$ where $\mu_{p^n}$ is the set of $p^n$th roots of unity and $K_\infty$ denotes $\cup_{n\ge1} K_n$. The $\Zp$-cyclotomic extension of $K$ is denoted by $K_c$ and $K^{(n)}$ denotes the $p^n$-subextension inside $K_c$. 

In particular, we write $\Qpn=\Qp(\mu_{p^n})$. For $n\ge m$, we write $\Tr_{n/m}$ for the trace map from $\Qpn$ to $\QQ_{p,m}$. For each $n$, we fix a primitive $p^n$th root of unity such that $\zeta_{p^n}^p=\zeta_{p^{n-1}}$. Let $G_n$ denote the Galois group $\Gal(\Qpn/\Qp)$ for $0\le n \le\infty$. Then, $G_\infty\cong\Delta\times\Gamma$ where $\Delta=G_1$ is a finite group of order $p-1$ and $\Gamma=\Gal(\QQ_{p,\infty}/\QQ_{p,1})\cong\Zp$. We fix a topological generator $\gamma$ of $\Gamma$ and write $u=\chi(\gamma)$. In particular, $u$ is a topological generator of $1+p\Zp$.


\subsection{Iwasawa algebras and power series}
 Given a finite extension $K$ of $\Qp$, $\Lambda_{\calO_K}(G_\infty)$ (respectively $\Lambda_{\calO_K}(\Gamma)$) denotes the Iwasawa algebra of $G_\infty$ (respectively $\Gamma$) over $\calO_K$. We write $\Lambda_K(G_\infty)=\Lambda_{\calO_K}(G_\infty)\otimes K$ and $\Lambda_K(\Gamma)=\Lambda_{\calO_K}(\Gamma)\otimes K$. When $K=\Qp$ (so $\calO_K=\Zp$), we simply write $\Lambda$ for $\Lambda_{\Zp}$. If $M$ is a finitely generated $\Lambda_{\calO_K}(\Gamma)$-torsion (respectively $\Lambda_K(\Gamma)$-torsion) module, we write $\Char_{\Lambda_{\calO_K}(\Gamma)}(M)$ (respectively $\Char_{\Lambda_K(\Gamma)}(M)$) for its characteristic ideal.

Given a module $M$ over $\Lambda_{\calO_K}(G_\infty)$ (respectively $\Lambda_K(G_\infty)$) and a character $\delta:\Delta\rightarrow\Zp^\times$, $M^\delta$ denotes the $\delta$-isotypical component of $M$. For any $m\in M$, we write $m^\delta$ for the projection of $m$ into $M^\delta$. The Pontryagin dual of $M$ is written as $M^\vee$.

Let $r\in\RR_{\ge0}$. We define
\[
\HH_r=\left\{\sum_{n\geq0,\sigma\in\Delta}c_{n,\sigma}\cdot\sigma\cdot X^n\in\CC_p[\Delta][[X]]:\sup_{n}\frac{|c_{n,\sigma}|_p}{n^r}<\infty\ \forall\sigma\in\Delta\right\}
\]
where $|\cdot|_p$ is the $p$-adic norm on $\CC_p$ such that $|p|_p=p^{-1}$ (the corresponding valuation is written as $v_p$). We write $\HH_\infty=\cup_{r\ge0}\HH_r$ and
$\HH_r(G_\infty)=\{f(\gamma-1):f\in\HH_r\}$
 for $r\in\RR_{\ge0}\cup\{\infty\}$. In other words, the elements of $\HH_r$ (respectively $\HH_r(G_\infty)$) are the power series in $X$ (respectively $\gamma-1$) over $\CC_p[\Delta]$ with growth rate $O(\log_p^r)$. If $F,G\in\HH_\infty$ or $\HH_\infty(G_\infty)$ are such that $F=O(G)$ and $G=O(F)$, we write $F\sim G$.

Given a subfield $K$ of $\CC_p$, we write $\HH_{r,K}=\HH_r\cap K[\Delta][[X]]$ and similarly for $\HH_{r,K}(G_\infty)$. In particular, $\HH_{0,K}(G_\infty)=\Lambda_{K}(G_\infty)$. Moreover, we have three operators $\vp$, $\partial$ and $\psi$ on $\HH_{r,K}$ defined by
\[
\vp(f)=f((1+X)^p-1),\quad\partial f=(1+X)\frac{df}{dX}\quad\text{and}\quad\psi(f)=\sum_{\zeta^p=1}f(\zeta(1+X)-1).
\]


\subsection{Crystalline representations}
We write $\BB_{\cris}$ and $\BB_{\rm dR}$ for the rings of Fontaine and $\vp$ for the Frobenius acting on these rings. Recall that there exists an element $t\in\BB_{\rm dR}$ such that $\vp(t)=pt$ and $g\cdot t=\chi(g)t$ for $g\in G_{\Qp}$.

Let $V$ be a $p$-adic representation of $G_{\Qp}$ which is crystalline.  We denote the Dieudonn\'{e} module by $\DD(V)=\DD_{\cris}(V)=(\BB_{\cris}\otimes V)^{G_{\Qp}}$. If $j\in\ZZ$, $\DD^j(V)$ denotes the $j$th de Rham filtration of $\DD(V)$. 

We write $\DD_\infty(V)=\HH_{0,\Qp}^{\psi=0}\otimes\DD(V)$, which is contained in $\HH_{\infty,\Qp}\otimes\DD(V)$. The map $\vp\otimes\vp$ on $\HH_{\infty,\Qp}\otimes\DD(V)$ is simply written as $\vp$ and the map $\partial\otimes1$ is written as $\partial$. Note that $\partial$ acts on $\DD_\infty(V)$ bijectively, so $\partial^j$ makes sense for any $j\in\ZZ$.

Let $T$ be a lattice of $V$ which is stable under $G_{\Qp}$. For integers $m\ge n$, we write $\cor_{m/n}$ for the corestriction map $H^1(\QQ_{p,m},A)\rightarrow H^1(\Qpn,A)$ where $A=V$ or $T$. Let $\HIw(T)$ denote the inverse limit $\displaystyle\lim_{\leftarrow}H^1(\Qpn,T)$ with respect to the corestriction and $\HIw(V)=\QQ\otimes\HIw(T)$. Moreover, if $V$ arises from the restriction of a $p$-adic representation of $G_\QQ$ and $T$ is a lattice stable under $G_\QQ$, we write
\begin{eqnarray*}
\Hh^1(T)&=&\lim_{\stackrel{\longleftarrow}{n}}H^1(\ZZ[\zeta_{p^n},1/p],T),\\
\Hh^1(V)&=&\QQ\otimes\Hh^1(T).
\end{eqnarray*}

Let $V(j)$ denote the $j$th Tate twist of $V$, i.e. $V(j)=V\otimes\Qp e_j$ where $G_{\Qp}$ acts on $e_j$ via $\chi^j$. We have $\DD(V(j))=t^{-j}\DD(V)\otimes e_j$. For any $v\in\DD(V)$, $v_j=v\otimes t^{-j}e_j$ denotes its image in $\DD(V(j))$. We write $\Tw_{j,V}:\HIw(V)\rightarrow\HIw(V(j))$ for the isomorphism defined in \cite[Section A.4]{perrinriou93}, which depends on our choice of $\zeta_{p^n}$. For each $n$ and $j$, we write
\[
\exp_{n,j}:\Qpn\otimes\DD(V(j))\rightarrow H^1(\Qpn,V(j))
\]
for Bloch-Kato's exponential defined in \cite{blochkato}.


\subsection{Modular forms}\label{modularforms}

   Let $f=\sum a_nq^n$ be a normalised eigen-newform of weight $k\ge2$, level $N$ and character $\epsilon$. Write $F_f=\QQ(a_n:n\ge1)$ for its coefficient field. Let $\bar{f}=\sum\bar{a}_nq^n$ be the dual form to $f$, we have $F_f=F_{\bar{f}}$.

We write $L(f,s)$ for the complex $L$-function of $f$. If $\theta$ is a finite character of $G_\infty$, we write $L(f_\theta,s)$ for the twisted $L$-function of $f$ by $\theta$.

 We assume that $p\nmid N$ and fix a prime of $F_f$ above $p$. We denote the completion of $F_f$ at this prime by $E$ and fix a uniformiser $\varpi$. We write $V_f$ for the 2-dimensional $E$-linear representation of $G_{\QQ}$ associated to $f$ from \cite{deligne69}. When restricted to $G_{\Qp}$, $V_f$ is crystalline and its de Rham filtration is given by
    \begin{equation}\label{filtration}
     \DD^i(V_f)=
     \left\{
     \begin{array}{ll}
      \DD(V_f)         & \text{if $i\le0$}\\
      E\omega                     & \text{if $1\le i\le k-1$}\\
      0                          & \text{if $i\ge k$}
     \end{array}\right.
    \end{equation}
for some $0\ne\omega\in\DD(V_f)$. Hence, the Hodge-Tate weights of $V_f$ are $0$ and $1-k$. The action of $\vp$ on $\DD(V_f)$ satisfies $\vp^2-a_p\vp+\epsilon(p)p^{k-1}=0$.

If $v\in V_f$, we write $v^\pm$ for the component of $v$ on which the complex conjugation acts by $\pm1$.


\section{Construction of the Coleman maps}\label{ap0}

In this section, we define the plus and minus Coleman maps for a modular form $f$ as in Section~\ref{modularforms} under the following condition:
\begin{itemize}
\item {\bf Assumption (1)}: $a_p=0$ and the eigenvalues of $\vp$ on $\DD(V_f)$ are not integral powers of $p$.
\end{itemize}

 We first review the definition of Perrin-Riou's exponential from \cite{perrinriou94} for general crystalline representations and results of Kato \cite{kato04} on general modular forms. We then prove a divisibility property of the image of the Perrin-Riou pairing under assumption (1) in order to define $\col^\pm$.


\subsection{Perrin-Riou's exponential}
Throughout this section, we fix $V$ a crystalline $p$-adic representation of $G_{\Qp}$ such that the action of $\vp$ on $\DD(V)$ has no eigenvalues which are integral powers of $p$. Let $j$ be an integer. Since $\vp$ acts on $t$ via multiplication by $p$ and $\DD(V(j))=t^{-j}\DD(V)\otimes e_j$, the eigenvalues of $\vp$ on $\DD(V(j))$ are not integral powers of $p$ either.

Since $V(j)^{G_{\Qpi}}$ is also a crystalline representation, it is a sum of characters. But a character is crystalline if and only if it is the product of an unramified character and a power of $\chi$ (see for example \cite[Example 3.1.4]{breuil}). Therefore, our assumption on the eigenvalues of $\vp$ implies that $V(j)^{G_{\Qpi}}=0$. 

For each $j\in\ZZ$ and $n\ge0$, under our assumptions on the eigenvalues of $\vp$, the exponential map $\exp_{n,j}$ induces an isomorphism
\[
\exp_{n,j}:\Qpn\otimes \DD(V(j))/\DD^{0}(V(j))\rightarrow H^1_f(\Qpn,V(j)).
\]
When $n\ge1$, there is a well-defined map
\begin{eqnarray*}
\Xi_{n,V(j)}:\DD_\infty(V(j))&\rightarrow&\Qpn\otimes\DD(V(j))\\
g&\mapsto&(p\otimes\vp)^{-n}G(\zeta_{p^n}-1)
\end{eqnarray*}
where $G\in \HH_{\infty,\Qp}\otimes\DD(V)$ is such that $(1-\vp)G=g$ (see \cite[Section~3.2.2]{perrinriou94}). Moreover, $(\exp_{n,j}\circ\Xi_{n,V(j)})_{n\geq1}$ are compatible with the corestriction maps. In other words, the following diagram commutes:
\[
\xymatrix{
\DD_\infty(V(j))\ar[rrrrd]_{\exp_{n,j}\circ\Xi_{n,V(j)}}\ar[rrrr]^{\exp_{n+1,j}\circ\Xi_{n+1,V(j)}}&&&&H^1(\QQ_{p,n+1},V(j))\ar[d]^{\mathrm{cor}_{n+1/n}}\\
&&&&H^1(\Qpn,V(j)).
}
\]

The definition of the Perrin-Riou exponential is given by the following theorem, which is the main result of \cite{perrinriou94}.

\begin{theorem}\label{prexp}
Let $h$ be a positive integer such that $\DD^{-h}(V)=\DD(V)$. Then, for all integers $j\geq1-h$, there is is a unique family of $\Lambda(G_\infty)$-homomorphisms
\[
\Omega_{V(j),h+j}:\DD_\infty(V(j))\rightarrow\HH_\infty(G_\infty)\underset{\Lambda(G_\infty)}{\otimes}\HIw(T(j))
\]
such that the following diagram commutes:
\[
\xymatrix{
\DD_\infty(V(j))\ar[rrr]^{\Omega_{V(j),h+j}\ \ \ \ \ \ \ }\ar[d]_{\Xi_{n,V(j)}}&&&\HH_\infty(G_\infty)\underset{\Lambda(G_\infty)}{\otimes}\HIw(T(j))\ar[d]^{\mathrm{pr}}\\
\Qpn\otimes\DD(V(j))\ar[rrr]^{(h+j-1)!\exp_{n,j}}&&& H^1(\Qpn,V(j))
}
\]
where $n\ge1$ and $\rm pr$ stands for projection. Moreover, we have
\[
\Tw_{1,V(j)}\circ\Omega_{V(j),h+j}\circ(\partial\otimes te_{-1})=-\Omega_{V(j+1),h+j+1}.
\]
\end{theorem}
\begin{proof}\cite[Section 3.2.3]{perrinriou94}\end{proof}

\begin{remark}\label{growthrate}
By \cite[Section 3.2.4]{perrinriou94}, if $g\in \HH_{0,\Qp}^{\psi=0}\otimes \DD_\alpha(V(j))$ where $\DD_\alpha(V(j))$ is the subspace of $\DD(V(j))$ in which $\vp$ has slope $\alpha$, then $\Omega_{V(j),h+j}(g)$ is $O(\log_p^{h+\alpha})$, i.e. contained in $\HH_{h+\alpha}(G_\infty)\otimes\HIw(T(j))$.
\end{remark}

\begin{remark}\label{con1}The theorem implies the following congruence for $r\ge0$:
\[(-1)^r\Tw_{r,V(j)}(\Omega_{V(j),h+j}(g))\equiv
 (h+j+r-1)!\exp_{n,j+r}\circ\Xi_{n,V(j+r)}\circ(\partial^{-r}\otimes t^{-r}e_r)(g)\mod (\gamma^{p^{n-1}}-1).
\]
\end{remark}

\subsection{Perrin-Riou's pairing}\label{prpair}
Let $M$ be a finite extension of $\Qp$ and we further assume that $V$ is a vector space over $M$ and the action of $G_{\Qp}$ is compatible with the multiplication by $M$. We fix $T$ an $\OO_M$-lattice of $V$ which is stable under $G_{\Qp}$. We write $V^*$ for the $M$-linear dual of $V$ and $T^*$ for the $\OO_M$-linear dual of $T$. Since $H^1(\QQ_{p,n},T)$ and $H^1(\QQ_{p,n},T^*(1))$ are $\OO_M[G_n]$-modules, $\HIw(T)$ and $\HIw(T^*(1))$ are $\Lambda_M(G_\infty)$-modules. By \cite[Section 3.6.1]{perrinriou94}, there is a non-degenerate pairing
\begin{eqnarray*}
<,>:\HIw(T)\times\HIw(T^*(1))&\rightarrow&\Lambda_{\calO_M}(G_\infty)\\
((x_n)_n,(y_n)_n)&\mapsto&\left(\sum_{\sigma\in G_n}[x_n^\sigma,y_n]_n\cdot\sigma\right)_n
\end{eqnarray*}
where $[,]_n$ is the natural pairing
\[H^1(\Qpn,T)\times H^1(\Qpn,T^*(1))\rightarrow\calO_M.\]
The pairing $<,>$ extends to
\[
\left(\HH_{\infty,M}(G_\infty)\underset{\Lambda_{\calO_M}(G_\infty)}{\otimes}\HIw(T)\right)\times\left(\HH_{\infty,M}(G_\infty)\underset{\Lambda_{\calO_M}(G_\infty)}{\otimes}\HIw(T^*(1))\right)\rightarrow\HH_{\infty,M}(G_\infty),
\]
which we also denote by $<,>$. Let $j$ and $h$ be integers satisfying conditions of Theorem~\ref{prexp}. If $\eta\in \DD(V(j))$, then $(1+X)\otimes\eta\in\DD_\infty(V(j))$. Using the pairing $<,>$, we define a map:
\begin{eqnarray*}
\LL_\eta^{h,j}:\HIw(T(j)^*(1))&\rightarrow&\HH_{\infty,M}(G_\infty)\\
\mathbf{z}&\mapsto&<\Omega_{V(j),h+j}((1+X)\otimes\eta),\mathbf{z}>.
\end{eqnarray*}
Note that $\LL_\eta^{h,j}$ modulo $\gamma^{p^{n-1}}-1$ induces a map into $M[G_n]$, which we denote by $\LL_{\eta,n}^{h,j}$. Also, $\LL_\eta^{h,j}$ extends naturally to a map on $\HIw(V(j)^*(1))$, which we write as $\LL_\eta^{h,j}$ also.


\subsubsection{Explicit formulae of $\LL_{\eta,n}^{h,j}$} We want to say something about values of the image of $\LL_{\eta,n}^{h,j}$ at some special characters on $G_\infty$. To do this, we make use of the following result. 

  \begin{lemma}
   Under the notation above, let $\eta\in\DD(V(j))$. Then, the projection of 
\[
\frac{1}{(h+j-1)!}\Omega_{V(j),h+j}((1+X)\otimes\eta)
\]  
 into $H^1(\Qpn,V(j))$ is given by
  \[
     \left\{
      \begin{array}{ll}
        p^{-n}\exp_{n,j}\left(\sum_{m=0}^{n-1}\zeta_{p^{n-m}}\otimes\vp^{m-n}(\eta)+(1-\vp)^{-1}(\eta)\right)    & \text{if $n\geq 1$}\\
            \exp_{0,j}\left(\left(1-\frac{\vp^{-1}}{p}\right)(1-\vp)^{-1}(\eta)\right)           & \text{if $n=0$.}
      \end{array}\right.
    \]
   \end{lemma}
   \begin{proof}
    This is a straightforward application of Remark~\ref{con1} to the solution of $(1-\vp)G=(1+X)\otimes\eta$ as given in \cite[Section 2.2]{perrinriou94}.
   \end{proof}

For $n\ge1$ and $\eta\in\DD(V(j))$, we write
\[
\gamma_{n,j}(\eta):=p^{-n}\left(\sum_{i=0}^{n-1}\zeta_{p^{n-i}}\otimes\vp^{i-n}(\eta)+(1-\vp)^{-1}(\eta)\right).
\]
Remark \ref{con1} and properties of the twist map (see e.g. \cite[Sections 3.6.1 and 3.6.5]{perrinriou94}) implies that for $\mathbf{z}\in\HIw(T(j)^*(1))$ and $r\ge0$,
\begin{equation}\label{con2}
\frac{1}{(h+j+r-1)!}\Tw_{r}(\LL^{h,j}_\eta(\mathbf{z}))
\equiv\sum_{\sigma\in G_n}\left[\exp_{n,j+r}(\gamma_{n,j+r}(\eta_r)^\sigma),z_{-r,n}\right]_n\cdot\sigma\mod (\gamma^{p^{n-1}}-1)
\end{equation}
where $\Tw_r$ acts on $\HH_\infty(G_\infty)$ via $\sigma\mapsto\chi(\sigma)^r\sigma$ for $\sigma\in G_\infty$ and $z_{-r,n}$ is the image of $\mathbf{z}$ under the composition
\[
\HIw(T(j)^*(1))\stackrel{(-1)^r\Tw_{-r}}{\xrightarrow{\hspace*{1.5cm}}}\HIw(T(j+r)^*(1))\stackrel{\rm{pr}}{\longrightarrow}H^1(\Qpn,T(j+r)^*(1)).
\]
By \cite[Chapter II, Section 1.4]{kato93}, we also have
\[\left[\exp_{n,j+r}(\cdot),\cdot\right]_n=\Tr_{n/0}\otimes\id\left(\left[\cdot,\exp^*_{n,j+r}(\cdot)\right]_n'\right)\]
where $\exp^*_{n,j+r}$ is the dual exponential map
\[
\exp^*_{n,j+r}:H^1(\Qpn,V(j+r)^*(1))\rightarrow \DD^0(V(j+r)^*(1))
\]
and the pairing
\[
[,]_n':\Qpn\otimes \DD(V(j+r))\times \Qpn\otimes \DD(V(j+r)^*(1))\rightarrow \Qpn\otimes M
\]
is induced by the natural pairing 
\[\DD(V(j+r))\times \DD(V(j+r)^*(1))\rightarrow M.\]
To ease notation, we simply write $[,]_n$ for $[,]_n'$ when it does not cause confusion. We can now rewrite (\ref{con2}) as:
\begin{equation}\label{con3}
\begin{split}
&\frac{1}{(h+j+r-1)!}\Tw_r(\LL_\eta^h(\mathbf{z}))\\
\equiv\ &\sum_{\sigma\in G_n}\Tr_{n,0}\left[\gamma_{n,j+r}(\eta_r)^\sigma,\exp^*_{n,j+r}(z_{-r,n})\right]_n\cdot\sigma\mod (\gamma^{p^{n-1}}-1)\\
\equiv\ &\left[\sum_{\sigma\in G_n}\gamma_{n,j+r}(\eta_r)^\sigma\sigma,\sum_{\sigma\in G_n}\exp^*_{n,j+r}(z_{-r,n}^\sigma)\sigma^{-1}\right]_n\mod (\gamma^{p^{n-1}}-1).
\end{split}
\end{equation}
Note that we have recovered the pairing $P_n$ of \cite{kurihara02}. We write the quantity in (\ref{con3}) as $P_{n,r}(\eta,z_{-r,n})$. Following the calculations of \cite{kurihara02}, we can deduce the following special values of $\LL_\eta^{h,j}$:

\begin{lemma}\label{char}
 For an integer $r\ge0$, we have
\begin{eqnarray*}
&&\frac{1}{(h+j+r-1)!}\chi^r\left(\LL_\eta^{h,j}(\mathbf z)\right)\\
&=&\left[\left(1-\frac{\vp^{-1}}{p}\right)(1-\vp^{-1})(\eta_r),\exp_{0,r+j}^*(z_{-r,0})\right]_0.
\end{eqnarray*}
Let $\theta$ be a character of $G_n$ which does not factor through $G_{n-1}$ with $n\ge1$, then
\begin{eqnarray*}
&&\frac{1}{(h+j+r-1)!}\chi^r\theta\left(\LL_\eta^{h,j}(\mathbf z)\right)\\
&=&\frac{1}{\tau(\theta^{-1})}\sum_{\sigma\in G_n}\theta^{-1}(\sigma)\left[\vp^{-n}(\eta_r),\exp_{n,r+j}^*(z_{-r,n}^\sigma)\right]_n
\end{eqnarray*}
where $\tau$ denotes the Gauss sum.
\end{lemma}


\subsection{Modular forms and Kato zeta elements}\label{katoreview}
The details of the results in this section can be found in \cite{kato04}.

\subsubsection{$L$-functions and $p$-adic $L$-functions}\label{lfunctions}
Let $f$ be as in Section~\ref{modularforms}. For any $v\in V_f$ such that $v^\pm\ne0$, it determines a lattice $\calO_E$-lattice $T_f$ of $V_f$. We choose $v$ such that $T_f$ is stable under $G_\QQ$. Note that as a representations of $G_\QQ$, $V_f^*\cong V_{\bar{f}}(k-1)$. Hence, $T_{f}$ determines a lattice $T_{\bar{f}}$ of $V_{\bar{f}}$ naturally.

Let ${\rm per}:\DD^1(V_f)\rightarrow V_f$ be the period map defined in \cite{kato04}. Fix $0\ne\omega\in\DD^1(V_f)$ and let $\Omega_{\pm}\in\CC^\times$ such that ${\rm per}(\omega)=\Omega_+v^++\Omega_-v^-$. The $p$-adic $L$-functions associated to $f$ are given by the following.

\begin{theorem}\label{padiclfunctions}
Let $\alpha$ be a root of $X^2-a_pX+\epsilon(p)p^{k-1}$ such that $v_p(\alpha)<k-1$. Under the notation above, there exists a unique $L_{p,\alpha}\in\HH_{\infty}(G_\infty)$ (depending on the choice of $\omega$ and $v$) such that for any integer $0\le r\le k-2$ and any character $\theta$ of $G_n$ which does not factor through $G_{n-1}$ with $n\ge1$,
\[
\chi^r\theta(L_{p,\alpha})=\frac{c_{n,r}\alpha^{-n}}{\tau(\theta)\Omega_{\pm}}L(f,\theta,r)
\]
where $c_{n,r}$ is some constant, only dependent on $n$ and $r$ and $\pm=(-1)^{k-r}\theta(-1)$.
\end{theorem}
\begin{proof}
\cite{amicevelu75}, \cite{MTT} or \cite[Theorem 16.2]{kato04}.
\end{proof}

If $f$ corresponds to an elliptic curve $\E$ over $\QQ$, there is a canonical choice of $\omega$ and $T_f$, namely, the N\'{e}ron differential and $T_p(\E)(-1)$ (see \cite[Section~2.2.2]{kurihara02}) where $T_p(\E)$ denotes the Tate module of $\E$ at $p$. 


\subsubsection{Kato's main conjecture}
In order to state Kato's main conjecture, we have to review two important results from \cite{kato04} first.
\begin{theorem}\label{katoglobal}Under the notation above, we have:
\begin{itemize}
\item[(a)]$\Hh^2(T_f)$ is a torsion $\Lambda_{\calO_E}(G_\infty)$-module.
\item[(b)]$\Hh^1(T_f)$ is a torsion free $\Lambda_{\calO_E(G_\infty)}$-module and $\Hh^1(V_f)$ is a free $\Lambda_{E}(G_\infty)$-module of rank 1.
\end{itemize}
\end{theorem}
\begin{proof}\cite[Theorem 12.4]{kato04} \end{proof}

\begin{theorem}\label{katozetamc}Fix a character $\delta:\Delta\rightarrow\ZZ/(p-1)\ZZ$.
\begin{itemize}
\item[(a)]Let $\theta$ be a character of $G_n$ and $\pm=(-1)^{k-r}\theta(-1)$ where $r$ is an integer such that $1\le r\le k-1$. Write
\begin{eqnarray*}
\kappa_\theta:\Qpn\otimes\DD^{0}(V_f(k-r))&\rightarrow&V_f\\
x\otimes y&\mapsto& \sum_{\sigma\in G_n}\theta(\sigma)\sigma(x){\rm per}(y)^{\pm}.
\end{eqnarray*}
There exists a unique $E$-linear map (independent of $\theta$ and $r$) $V_f\rightarrow\Hh^1(V_f)$; $v\mapsto\mathbf{z}_v$ such that $\kappa_\theta$ sends the image of $\mathbf{z}_v$ in $\Qpn\otimes\DD^0(V_f(k-r))$ (under the composition of the localisation, the twist map and the dual exponential) to $d_{r}\cdot L(\bar{f},\theta,r)\cdot v^{\pm}$ and $d_r$ is a constant which only depends on $r$.
\item[(b)]Let $\ZZ(T_f)\subset\Hh^1(V_f)$ denote the $\Lambda_{\calO_E}(G_\infty)$-module generated by $\mathbf{z}_{v^\pm}\in T_f$ and write $\ZZ(V_f)=\ZZ(T_f)\otimes\QQ$. Then, the quotient $\Hh^1(V_f)/\ZZ(V_f)$ is a torsion $\Lambda_{E}(G_\infty)$-module and 
\[\Char_{\Lambda_E(\Gamma)}(\Hh^1(V_f)^\delta/\ZZ(V_f)^\delta)\subset\Char_{\Lambda_E(\Gamma)}(\Hh^2(V_f)^\delta).\]
\item[(c)]If the homomorphism $G_\QQ\rightarrow GL_{\calO_E}(T_f)$ is surjective, then $\ZZ(T_f)\subset\Hh^1(T_f)$. Moreover, $\Hh^1(T_f)$ is a free $\Lambda_{\calO_E}$-module of rank 1 and
\[\Char_{\Lambda_{\calO_E}(\Gamma)}(\Hh^1(T_f)^\delta/\ZZ(T_f)^\delta)\subset\Char_{\Lambda_{\calO_E}(\Gamma)}(\Hh^2(T_f)^\delta).\] 
\end{itemize}
\end{theorem}
\begin{proof}\cite[Theorem 12.5]{kato04} \end{proof}

Kato's main conjecture states that:

\begin{conjecture}
The inclusion $\ZZ(T_f)\subset\Hh^1(T_f)$ holds. Moreover, if $\delta:\Delta\rightarrow\ZZ/(p-1)\ZZ$ is a character, then 
\[\Char_{\Lambda_{\calO_E}(\Gamma)}(\Hh^1(T_f)^\delta/\ZZ(T_f)^\delta)=\Char_{\Lambda_{\calO_E}(\Gamma)}(\Hh^2(T_f)^\delta).\]
\end{conjecture}

We call elements of $\ZZ(V_f)$ Kato zeta elements. In particular, we write $\kato_f$ for the one corresponding to our choice of ${v}\in V_{f}$ fixed in Section~\ref{lfunctions} and call it \textit{the} Kato zeta element associated to $f$.

We fix $\bar{v}\in V_{\bar{f}}$ and $\bar{\omega}\in\DD^{-1}(V_{\bar{f}}(k))$ for the dual form $\bar{f}$ similarly. Below, we relate the Kato zeta element $\kato_{\bar{f}}$ associated to $\bar{f}$ to the $p$-adic $L$-functions of $f$ defined by Theorem~\ref{padiclfunctions} via the map $\LL_{\eta}^{h,j}$. For simplicity, we write $\kato=\kato_{\bar{f}}$ from now on.

Let $V=V_f(1)$, then we can take $h=1$ and $j\ge0$ in Theorem~\ref{prexp} by \eqref{filtration}. For $\eta\in\DD(V_f)$, we simply write
\[
\LL_\eta=\LL_{\eta_1}^{1,0}:\HIw(T_{\bar{f}}(k-1))\rightarrow\HH_\infty(G_\infty)
\]
for the map we defined in Section~\ref{prpair}, with $M=E$.

\begin{theorem}\label{ka1}
For $\alpha$ as in Theorem~\ref{padiclfunctions}, there exists $\eta_\alpha$, an eigenvector of $\vp$ on $\DD(V_f)$ with eigenvalue $\alpha$ such that $[\eta_\alpha,\bar{\omega}]=1$. Moreover, the image of $\kato$ under the composition
\[
\Hh^1(V_{\bar{f}})\rightarrow\HIw(V_{\bar{f}})\stackrel{\Tw_{k-1}}{\longrightarrow}\HIw(V_{\bar{f}}(k-1))\stackrel{\LL_{\eta_\alpha}}{\longrightarrow}\HH_\infty(G_\infty)
\]
is the $p$-adic $L$-function $L_{p,\alpha}$ where the first map is just the localisation and $\Tw_{k-1}$ denotes $\Tw_{k-1,V_{\bar{f}}}$.
\end{theorem}
\begin{proof}\cite[Theorem 16.6]{kato04}\end{proof}

We sometimes abuse notation and write the above composition as $\LL_{\eta_\alpha}$ also. 

\begin{remark}\label{log}Let $\alpha_1$ and $\alpha_2$ be the roots of $X^2-a_pX+\epsilon(p)p^{k-1}$. Then, the slope of $\vp$ on $\DD(V_f)$ is equal to $t=\max(v_p(\alpha_1),v_p(\alpha_2))$. Since $h=1$ and the slope of $\vp$ on $\DD(V_f(1))$ is $t-1$, all elements of $\image(\LL_\eta)$ are $O(\log_p^t)$ by Remark~\ref{growthrate}.\end{remark}

It follows immediately from Lemma~\ref{char} that, with the same notation as in the lemma, we have:
\begin{equation}\label{char2}
\begin{split}
\chi^r(\LL_\eta(\mathbf{z}))&=r!\left[\left(1-\frac{\vp^{-1}}{p}\right)(1-\vp)^{-1}(\eta_{r+1}),\exp_{0,r+1}^*(z_{-r,0})\right]_0,\\
\chi^r\theta(\LL_\eta(\mathbf{z}))&=\frac{r!}{\tau(\theta^{-1})}\sum_{\sigma\in G_n}\theta^{-1}(\sigma)\left[\vp^{-n}(\eta_{r+1}),\exp_{n,r+1}^*(z_{-r,n}^\sigma)\right]_n.\\
\end{split}
\end{equation}


\subsection{The $\pm$-Coleman maps}

\subsection{$\pm$-logarithms}
Let $f$ be as above such that assumption (1) holds. If $\alpha_1$ and $\alpha_2$ are the roots of $X^2-a_pX+\epsilon(p)p^{k-1}$, then $\alpha_1=-\alpha_2$. Moreover, $v_p(\alpha_1)=v_p(\alpha_2)=(k-1)/2$, so Remark \ref{log} implies that $\image(\LL_{\eta})\subset\HH_{(k-1)/2}(G_\infty)$ for any $\eta\in\DD(V_f)$.

In \cite{pollack03}, Pollack defines:
\begin{eqnarray*}
\log_{p,k}^+&=&\prod_{j=0}^{k-2}\frac{1}{p}\prod_{n=1}^\infty\frac{\Phi_{2n}(u^{-j}\gamma)}{p},\\
\log_{p,k}^-&=&\prod_{j=0}^{k-2}\frac{1}{p}\prod_{n=1}^\infty\frac{\Phi_{2n-1}(u^{-j}\gamma)}{p},
\end{eqnarray*}
where $\Phi_m$ denotes the $p^m$th cyclotomic polynomial.

By considering the special values of $L_{p,\alpha_1}$ and $L_{p,\alpha_2}$ as given by Theorem~\ref{padiclfunctions}, Pollack shows that we have the following divisibility properties:
\begin{eqnarray*}
\log_{p,k}^+&|&\alpha_2L_{p,\alpha_1}-\alpha_1L_{p,\alpha_2}, \\
\log_{p,k}^-&|&L_{p,\alpha_2}-L_{p,\alpha_1}.
\end{eqnarray*}
This enables him to define
\begin{eqnarray}
L_{p,f}^+&=&\frac{\alpha_2L_{p,\alpha_1}-\alpha_1L_{p,\alpha_2}}{(\alpha_2-\alpha_1)\log_{p,k}^+},\label{pdef}\\
L_{p,f}^-&=&\frac{L_{p,\alpha_2}-L_{p,\alpha_1}}{(\alpha_2-\alpha_1)\log_{p,k}^-}\label{mdef}.
\end{eqnarray}

To ease notation, we suppress the subscript $f$ and write $L_{p}^\pm$ for $L_{p,f}^\pm$. The growth rates of these elements are given by:

\begin{theorem}\label{pollack}
$\log_{p,k}^+\sim\log_{p,k}^-\sim\log_p^{\frac{k-1}{2}}$ and $L_p^\pm=O(1)$.
\end{theorem}

\begin{proof}
\cite[Lemma 4.5 and Theorem 5.1]{pollack03}
\end{proof}


\subsubsection{Definition of the Coleman maps}\label{shorthand}
Let us first introduce a shorthand. For $0\le r\le k-2$ and $x\in\DD(V_f(r+1))$, we write $x\mod\omega$ for the image of $x$ in the quotient $\DD(V_f(r+1))/E\cdot\omega_{r+1}$. If two elements $x$ and $y$ of $\DD(V_f(r+1))$ have the same image, we simply write $x\equiv y\mod\omega$.

\begin{lemma}\label{modomega}
Let $0\le r\le k-2$ be an integer. If $\theta$ is a finite character as in Lemma~\ref{char} and $\eta\in\DD(V_f)$, then $\vp^{-n}(\eta_{r+1})\equiv0\mod\omega$ implies that $\chi^r\theta(\LL_\eta(\mathbf{z}))=0$ for any $\mathbf{z}$.
\end{lemma}
\begin{proof}
We have
\[\image(\exp^*_{n,r+1})=\Qpn\otimes E\cdot \bar{\omega}_{-r-1}=\Qpn\otimes\DD^0(V_{\bar{f}}(k-1-r))\quad\text{and}\quad\DD^0(V_f(r+1))=E\cdot\omega_{r+1}.\]
Hence, the fact that $\DD^0(V_f(r+1))$ and $\DD^0(V_{\bar{f}}(k-1-r))$ are orthogonal complements of each other under $[,]$ and \eqref{char2} implies that $\chi^r\theta(\LL_\eta(\mathbf{z}))=0$ if $\vp^{-n}(\eta_{r+1})$ is a multiple of $\omega_{r+1}$. 
\end{proof}

Recall that $\LL_{\eta_{\alpha_i}}(\kato)=L_{p,\alpha_i}$ for $i=1,2$ by Theorem \ref{ka1}. Hence, if we write 
\[
\eta^+=\frac{\alpha_2\eta_{\alpha_1}-\alpha_1\eta_{\alpha_2}}{\alpha_2-\alpha_1}\qquad\mathrm{and}\qquad\eta^-=\frac{\eta_{\alpha_2}-\eta_{\alpha_1}}{\alpha_2-\alpha_1},
\]
then $\LL_{\eta^{\pm}}(\mathbf{z}^{\mathrm{Kato}})=\log_{p,k}^{\pm}L_p^\pm$ by \eqref{pdef}, \eqref{mdef} and the linearity of $\LL$. In fact, more is true:

\begin{proposition}\label{pm}
If $\mathbf{z}\in\HIw(T_{\bar{f}})$, then $\log_{p,k}^\pm|\LL_{\eta^\pm}(\mathbf{z})$ over $\HH_{\infty,E}(G_\infty)$.
\end{proposition}
\begin{proof}
Recall that $[\omega,\bar{\omega}]=0$, $[\eta_{\alpha_i},\bar{\omega}]=1$ and $\vp^2=\alpha_i^2$ on $\DD(V_f)$. Therefore, explicit calculation shows that $\eta_{\alpha_i}=(\vp(\omega)+\alpha_i\omega)/[\vp(\omega),\bar{\omega}]$ for $i\in\{1,2\}$. Hence,
\[
\eta^+=\frac{\vp(\omega)}{[\vp(\omega),\bar{\omega}]}\qquad\mathrm{and}\qquad\eta^-=\frac{\omega}{[\vp(\omega),\bar{\omega}]}.
\]
Let $r$ be an integer. Since $\vp^2=-\epsilon(p)p^{k-2r-3}$ on $\DD(V_f(r+1))$, we have
\begin{align*}
\vp^{-n}(\eta^+_{r+1})\equiv0\mod\omega&\text{ if $n$ is odd,}\\
\vp^{-n}(\eta^-_{r+1})\equiv0\mod\omega&\text{ if $n$ is even.} 
\end{align*}
Therefore, by Lemma~\ref{modomega} and \eqref{char2}, we have
\begin{align*}
\chi^r\theta(\LL_{\eta^+}(\mathbf{z}))=0&\text{ if $n$ is odd},\\
\chi^r\theta(\LL_{\eta^-}(\mathbf{z}))=0&\text{ if $n$ is even}
\end{align*}
where $\theta$ and $n$ are as defined in Lemma~\ref{char}. Recall that $\chi(\gamma)=u$, so we have equivalences $\chi^r\theta(\Phi_m(u^{-r}\gamma))=\Phi_m(\theta(\gamma))=0$ if and only if $\theta(\gamma)$ is a primitive $p^m$th root of unity if and only if $\theta$ factors through $G_{m+1}$ but not $G_m$. Hence all the zeros of $\log_{p,k}^\pm$, which are all simple, are also zeros of $\LL_{\eta^\pm}(\mathbf{z})$, so we are done.\end{proof}

\begin{remark}
An alternative proof for this proposition is given in Section~\ref{image}.
\end{remark}

Recall that $\LL_{\eta^{\pm}}(\mathbf{z})=O(\log_p^{\frac{k-1}{2}})$ and Theorem~\ref{pollack} says that $\log_{p,k}^\pm\sim\log_p^{\frac{k-1}{2}}$, so we have $\LL_{\eta^{\pm}}(\mathbf{z})/\log_{p,k}^\pm=O(1)$, i.e. an element of $\HH_{0,E}(G_\infty)=\Lambda_E(G_\infty)$. We define
\begin{eqnarray*}
\col^\pm:\HIw(T_{\bar{f}}(k-1))&\rightarrow&\Lambda_{E}(G_\infty)\\
\mathbf{z}&\mapsto&\frac{\LL_{\eta^{\pm}}(\mathbf{z})}{\log_{p,k}^\pm}.
\end{eqnarray*}
We call these two maps the plus and minus Coleman maps. Note that we sometimes abuse notation and write $\col^\pm$ for the composition 
\[\Hh^1(T_{\bar{f}})\rightarrow\HIw(T_{\bar{f}})\stackrel{\Tw_{k-1}}{\longrightarrow}\HIw(T_{\bar{f}}(k-1))\stackrel{\col^\pm}{\longrightarrow}\Lambda_{E}(G_\infty)\] and its natural extension to $\Hh^1(V_{\bar{f}})$. In particular, we have 
\begin{equation}\label{colemankato}
\col^\pm(\mathbf{z}^\textrm{Kato})=L_p^\pm.
\end{equation}
Similar to $\LL_{\eta^\pm,n}$,  we write $\col_n^\pm$ for the map $\col^\pm$ modulo $\gamma^{p^{n-1}}-1$.

\begin{remark}
The Coleman maps in \cite{kobayashi03} are defined using a pairing with points coming from the formal group associated to an elliptic curve, instead of images of the Perrin-Riou exponential. It is not hard to see that the definition given above agrees with the one given by Kobayashi on comparing \cite[Proposition~8.25]{kobayashi03} and \eqref{con3}.
\end{remark}


\section{Kernels of the Coleman maps}\label{kernel}

In addtion to assumption (1), we assume the following holds.
\begin{itemize}
\item {\bf Assumption (2)}: Either $p+1\nmid k-1$ or $\epsilon(p)\ne-1$.
\end{itemize}
Under these two conditions, we give an explicit description of the kernels of the plus and minus Coleman maps defined in Section~\ref{ap0}. In particular, we generalise \cite[Proposition~8.18]{kobayashi03}, which describe the kernels of $\col^\pm$ in the case of elliptic curves defined over $\QQ$.


\subsection{Some linear algebra}
Let us first study some basic properties of $\Qpn$. Define
\[
\pi_n=\begin{cases}
\zeta_{p^n}&\text{if $n>1$},\\
\zeta_p+\frac{1}{p-1}&\text{if $n=1$},\\
1&\text{if $n=0$}
\end{cases}
\]
and $\Qp^{(n)}$ denotes the $\Qp$-vector space generated by $\{\pi_n^{\sigma}:\sigma\in G_n\}$. Then, $\Tr_{n/n-1}\pi_n=0$ for $n\ge1$ and 
\begin{equation}\label{Qpndeco}
\Qpn=\bigoplus_{i=0}^n\Qp^{(i)}.
\end{equation}
\begin{proposition}\label{spanning}
Let $n\ge0$ be an integer and $\alpha=\sum_{i=0}^nx_i\pi_i$ for some $x_i\in\Qp$. Then, the $\Qp$-vector space generated by $\{\alpha^\sigma:\sigma\in G_n\}$ is given by $ \bigoplus_{i\in S}\Qp^{(i)}$ where $S=\{i:x_i\ne0\}$.
\end{proposition}
\begin{proof}
We proceed by induction on $|S|$. The case $|S|=1$ is immediate, so we assume $|S|>1$. Write $V$ for the $\Qp$-vector space generated by $\{\alpha^\sigma:\sigma\in G_n\}$. Clearly, $V\subset\bigoplus_{i:x_i\ne0}\Qp^{(i)}$.
Without loss of generality, we assume that $x_n\ne0$. Let $\beta=\sum_{i=0}^{n-1}x_i\pi_i$. Then, by induction,  $\{\beta^\tau:\tau\in G_{n-1}\}$ generates $\bigoplus_{i\in S\setminus\{n\}}\Qp^{(i)}$ over $\Qp$. Fix $\tau\in G_{n-1}$, then
\[
\sum_{\sigma\in G_n, \sigma|_{\QQ_{p,n-1}}=\tau}\alpha^\sigma=r\beta^\tau+(\Tr_{n/n-1}\pi_n)^\tau=r\beta^\tau\in V
\]
where $r=[\Qpn:\QQ_{p,n-1}]$. Therefore, for any $\tau\in G_{n-1}$, $\beta^\tau\in V$ and $\pi_n^\sigma\in V$ for any $\sigma\in G_n$. Hence we are done.\end{proof}

\begin{corollary}\label{gen}
Let $\eta=a_0+\sum_{i=1}^na_i\zeta_{p^i}$ where $a_i\in \Qp$ with $a_1\ne(p-1)a_0$, then the $\Qp$-vector space generated by $\{\eta^\sigma:\sigma\in G_n\}$ is given by
$\displaystyle
\Qp+\sum_{r\in S}\sum_{\sigma\in G_n}\Qp\cdot\zeta_{p^r}^\sigma
$
where $S=\{r\in[1,n]:a_r\ne0\}$.
\end{corollary}
\begin{proof}The result is immediate if $a_1=0$ by Proposition~\ref{spanning}. If $a_1\ne0$, then
\[
\eta=\left(a_0-\frac{a_1}{p-1}\right)+a_1\pi_1+\sum_{i>1}a_i\pi_i.
\]
Hence, we can again apply Proposition~\ref{spanning}.
\end{proof}

\begin{corollary}
Let $\eta=1+\zeta_{p}+\zeta_{p^2}+\cdots+\zeta_{p^n}$, then $\eta$ is a normal basis of $\Qpn$ over $\Qp$.
\end{corollary}


\subsection{Properties of $H^1$}

Recall that when $f$ corresponds to an elliptic curve $\E$ over $\QQ$ and $T_f(1)$ is the Tate module of $\E$, we have $\E[p^\infty]\cong V_f/T_f(1)$ as $G_\QQ$-modules. Therefore, the following lemma generalises \cite[Proposition~8.7]{kobayashi03}, which says that $\E$ has no $p$-torsion defined over $k_\infty$.

\begin{lemma}\label{inv}
For all $j\in\ZZ$ and $n\ge0$, $(V_f/T_f)(j)^{G_{\Qpn}}=0$.
\end{lemma}

\begin{proof}
It is enough to show that $(V_f/T_f)^{G_{\Qpi}}=0$. Since $\displaystyle V_f/T_f=\lim_{\stackrel{\longleftarrow}{\times\varpi}}T_f/\varpi^nT_f$, it in fact suffices to show that $(T_f/\varpi T_f)^{G_{\Qpi}}=0$. We make use of the description of the representation $\rho_f:G_{\Qpn}\rightarrow GL(T_f/\varpi T_f)$ given by \cite[Proposition~4.1.4]{bergerlizhu04} and consider two different cases.

\underline{Case 1, $p+1\nmid k-1$}: In this case, 
\[\rho_f|I=\begin{pmatrix}
\psi^{k-1}&0\\
0&\psi'^{k-1}
\end{pmatrix}\] where $I$ is the inertia group of $G_{\Qp}$ and $\psi$ and $\psi'$ are fundamental characters of level $2$, i.e.
\[
\ker\psi=\ker\psi'=G_{\Qp^{\text{ur}}(\sqrt[p^2-1]{p})}.
\]
Hence, $1$ is not an eigenvalue of $\rho_f(\sigma)$ for all $\sigma\in\mathrm{Gal}(\Qp^\textrm{ur}(\sqrt[p^2-1]{p})/\Qp^\textrm{ur}(\sqrt[p-1]{p}))$, as $p+1\nmid k-1$. Therefore, there exists an element in the above Galois group which lifts to $G_{\Qpi}$ and $(T_f/\varpi T_f)^{G_{\Qpi}}=0$ as required.

\underline{Case 2, $p+1|k-1$}: In this case, $\rho_f|_{G_{\Qpi}}$ factors through $\Gal(\Qpi^{\rm ur}/\Qpi)$ and the eigenvalues of the Frobenius are the sqaure roots of $-\epsilon(p)$. By our assumption, this is not 1, so we are done. 
\end{proof}

We now give two immediate corollaries.

\begin{corollary}\label{proj}
The projection $\HIw(T_{\bar{f}}(j))\rightarrow H^1(\Qpn,T_{\bar{f}}(j))$ is surjective for all $j$ and $n$.
\end{corollary}

\begin{proof}
It is enough to show that $\cor_{n/m}:H^1(\Qpn,T_{\bar{f}}(j))\rightarrow H^1(\QQ_{p,m},T_{\bar{f}}(j))$ is surjective for all $n\ge m$. On taking Pontryagin dual, it is equivalent to showing that 
\[{\rm res}_{m/n}:H^1(\QQ_{p,m},V_f/T_{f}(k-1-j))\rightarrow H^1(\Qpn,V_f/T_{f}(k-1-j))\]
is injective. But this immediately follows from the inflation-restriction exact sequence and the fact that $V_f/T_f(k-1-j)^{G_{\Qpi}}=0$ as given by Lemma~\ref{inv}. 
\end{proof}

\begin{corollary}\label{sublattice}
For all $n$ and $j$ as above, $H^1(\Qpn,T_f(j))\hookrightarrow H^1(\Qpn,V_f(j))$.
\end{corollary}

\begin{proof}
From the short exact sequence $0\rightarrow T_f(j)\rightarrow V_f(j)\rightarrow V_f/T_f(j)\rightarrow0$, we obtain a long exact sequence
\[
\cdots\rightarrow(V_f/T_f(j))^{G_{\Qpn}}\rightarrow H^1(\Qpn,T_f(j))\rightarrow H^1(\Qpn,V_f(j))\rightarrow\cdots.
\]
Hence the result by Lemma~\ref{inv}.
\end{proof}

In particular, $H^1(\Qpn,T_f(j))$ can be identified as an $\OO_E$-lattice of $H^1(\Qpn,V_f(j))$. Another property of $H^1$ which we need is the injectivity of the restriction 
\[
H^1(\QQ_{p,m},V_f(j))\stackrel{\mathrm{res}}{\longrightarrow}H^1(\Qpn,V_f(j))
\]
 for $n\ge m$, which follows from the inflation-restriction sequence and that $V_f(j)^{G_{\QQ_{p,\infty}}}=0$ (immediate from Lemma~\ref{inv}). In particular, the same can be said about $H^1_f$. We regard $H^1_f(\QQ_{p,m},A)$ as a subgroup of $H^1_f(\QQ_{p,n},A)$ for $A=T_f(j)$ or $V_f(j)$ in the next section.


\subsection{Some subgroups of $H_f^1$}

Let $\eta^\pm$ be as defined in Section~\ref{ap0}. For $1\le j\le k-1$, recall that $\DD^0(V_f(j))=E\cdot\omega_j$. Using the shorthand introduced in Section~\ref{shorthand}, we define two $E[G_n]$-modules
\begin{equation}\label{defineevenodd}
\begin{split}
R_{n,j}^+&=\sum_{\sigma\in G_n}E\cdot\gamma_{n,j}(\eta^+_j)^\sigma\mod\omega\subset\Qpn\otimes \DD(V_f(j))/\DD^0(V_f(j)),\\
R_{n,j}^-&=\sum_{\sigma\in G_n}E\cdot\gamma_{n,j}(\eta^-_j)^\sigma\mod\omega\subset\Qpn\otimes \DD(V_f(j))/\DD^0(V_f(j)).
\end{split}
\end{equation}

\begin{remark}\label{cortr}
For $1\le j\le k-1$, we have isomorphisms of $E[G_{n}]$-modules
\[
H^1_f(\Qpn,V_f(j))\cong\Qpn\underset{\Qp}{\otimes} \DD(V_f(j))/\DD^0(V_f(j))\cong \Qpn\otimes E.
\]
Under this identification, the corestriction $\cor_{n/m}:H^1_f(\Qpn,V_f(j))\rightarrow H^1_f(\QQ_{p,m},V_f(j))$ corresponds to $\Tr_{n/m}\otimes\id:\Qpn\otimes E\rightarrow\QQ_{p,m}\otimes E$.
\end{remark}

By Remark~\ref{cortr}, we can identify $R_{n,j}^\pm$ with subsets of $\Qpn\otimes E$ and we have the following description.

\begin{lemma}
By identifying $\Qpn\otimes \DD(V(j))/\DD^0(V(j))$ with $\Qpn\otimes E$, we have
\begin{equation}\label{evenodd}
\begin{split}
R_{n,j}^+&=\sum_{m\text{ even}}\sum_{\sigma\in G_m}E\cdot\zeta_{p^m}^\sigma+E,\\
R_{n,j}^-&=\sum_{m\text{ odd}}\sum_{\sigma\in G_m}E\cdot\zeta_{p^m}^\sigma+E
\end{split}
\end{equation}
where $m\le n$ in the summands.
\end{lemma}
\begin{proof}
Recall that $\gamma_{n,j}=p^{-n}\left(\sum_{i=0}^{n-1}\zeta_{p^{n-i}}\otimes\vp^{i-n}+(1-\vp)^{-1}\right)$
and $\eta^\pm$ are given by the following:
\[
\eta^+=\frac{\vp(\omega)}{[\vp(\omega),\bar{\omega}]}\qquad\text{and}\qquad\eta^-=\frac{\omega}{[\vp(\omega),\bar{\omega}]}.
\]
Hence, we can apply Corollary~\ref{gen} to $R_{n,j}^\pm$ provided that 
\[(p-1)(1-\vp)^{-1}(\eta^\pm_j)\not\equiv\vp^{-1}(\eta^\pm_j)\mod\omega,\] which can be checked under assumption (1). Recall that $\vp^m(\omega)\equiv0\mod\omega$ if and only if $m$ is an even integer (c.f. proof of Proposition~\ref{pm}), hence the result.
\end{proof}

In particular, \eqref{Qpndeco} and \eqref{evenodd} implies that
\[R_{n,j}^++R_{n,j}^-=\Qpn\otimes E\quad\text{and}\quad R_{n,j}^+\cap R_{n,j}^-=E\]
under the identification given by Remark~\ref{cortr}.
 Let
\[
\Qpn^\pm=\{x\in \Qpn:\Tr_{n/m+1}(x)\in \QQ_{p,m}\ \forall m\in S_n^\pm\}\]
where $S_n^\pm$ are defined by
\begin{eqnarray*}
S_n^+&=&\{m\in[0,n-1]:m\text{ even}\},\\
S_n^-&=&\{m\in[0,n-1]:m\text{ odd}\}.
\end{eqnarray*}
Then, $R_{n,j}^\pm$ can be identified with $\Qpn^\pm\otimes E$:
\begin{lemma}\label{coincides}
For $j$ and $n$ as above, $\Qpn^\pm\otimes E=R_{n,j}^\pm$.
\end{lemma}
\begin{proof}By \eqref{evenodd}, it is easy to check that $R_{n,j}^\pm\subset \QQ_{p,n}^\pm\otimes E$, so $\dim_{E}R_{n,j}^\pm\le\dim_{E}\left(\Qpn^\pm\otimes E\right)$.
Since $R_{n,j}^++R_{n,j}^-=\Qpn\otimes E$, we have 
\[
\Qpn^+\otimes E+\Qpn^-\otimes E=R_{n,j}^++R_{n,j}^-=\Qpn\otimes E.
\]

If $x\in \Qpn^+\cap \Qpn^-$, then $\Tr_{n/m+1}(x)\in \QQ_{p,m}$ for all $m\le n-1$, hence $x\in\Qp$. Therefore, we have $\Qpn^+\cap \Qpn^-=\Qp$. Hence, by the formula $\dim A+\dim B=\dim(A+B)+\dim(A\cap B)$, we deduce that $\dim_{E}\left(\Qpn^\pm\otimes E\right)=\dim_{E}R_{n,j}^\pm$ and we are done.\end{proof}

Let $H_f^1(\Qpn,V_f(j))^\pm$ denote the image of $R_{n,j}^{\pm}$ under $\exp_{n,j}$, then Remark~\ref{cortr} and Lemma~\ref{coincides} implies that it is equal to
\[
\left\{x\in H^1_f(\Qpn,V_f(j)):\mathrm{cor}_{n/m+1}(x)\in H^1_f(\Qpn,V_f(j))\ \forall m\in S_n^\pm\right\}.
\]

By Corollary~\ref{sublattice}, if we define
\[H^1_f(\Qpn,T_f(j))^\pm=H_f^1(\Qpn,V_f(j))^\pm\cap H^1_f(\Qpn,T_f(j)),\]
then it is equal to
\[
\left\{x\in H^1_f(\Qpn,T_f(j)):\mathrm{cor}_{n/m+1}(x)\in H^1_f(\QQ_{p,m},T_f(j))\ \forall m\in S_n^\pm\right\}
\]
generalising the definition of $E^\pm$ in \cite{kobayashi03}.


\subsection{Description of the kernels}

Let $\mathbf{z}\in \HIw(T_{\bar{f}}(k-1))$. Under the notation of Section~\ref{ap0}, we have $\LL_{\eta^\pm}(\mathbf{z})=O(\log_p^{\frac{k-1}{2}})$, so  we have $\LL_{\eta^\pm}(\mathbf{z})=0$ if and only if $P_{n,r}(\eta^\pm,z_{-r,n})=0$
 for all $n\ge0$ and more than $(k-1)/2$ different values of $r$ with $0\le r\le k-2$. Recall that
\[
P_{n,r}(\cdot,z_{-r,n})=r!\sum_{\sigma\in G_n}\left[\exp_{n,r+1}(\gamma_{n,r+1}(\cdot)^\sigma),z_{-r,n}\right]_n\sigma.
\]
Hence, $\ker P_{n,r}(\eta^{\pm},\cdot)$ is just the annihilator of $\left\{\exp_{n,r+1}(\gamma_{n,r+1}(\eta^\pm)^\sigma):\sigma\in G_n\right\}$ under the pairing
\[
H^1(\Qpn,V_f(r+1))\times H^1(\Qpn,T_{\bar{f}}(k-1-r))\rightarrow E
\]
which coincides with the annihilator of $H^1_f(\Qpn,T_f(r+1))^\pm$ under the pairing
\begin{equation}\label{pairing}
H^1(\Qpn,T_f(r+1))\times H^1(\Qpn,T_{\bar{f}}(k-1-r))\rightarrow \OO_E.
\end{equation}
We denote this annihilator by $H^1_\pm(\Qpn,T_{\bar{f}}(k-1-r))$.

Define $\Hpm(T_{\bar{f}}(k-1-r))=\displaystyle\lim_{\leftarrow}H^1_\pm(\Qpn,T_{\bar{f}}(k-1-r))$. As $\log_{p,k}^\pm\ne0$ and $\LL_{\eta^\pm}=\log_{p,k}^\pm\col^\pm$, Corollary~\ref{proj} implies that
\[
\ker\LL_{\eta^\pm}=\ker\left(\col^\pm\right)=\bigcap_{r=0}^{k-2}\Tw_r\left(\Hpm\left(T_{\bar{f}}(k-1-r)\right)\right).
\]
In fact, by the proposition below, it suffices to take just one term in the intersection.

\begin{proposition}$\Tw_r\left(\Hpm\left(T_{\bar{f}}(k-1-r)\right)\right)=\Hpm(T_{\bar{f}}(k-1))$ for all integers $r$ such that $0\le r\le k-2$.
\end{proposition}

\begin{proof}Since $\col^\pm(\mathbf{z})=O(1)$ for all $\mathbf{z}\in\HIw(T_{\bar{f}}(k-1))$, it is uniquely determined by its values at an infinite number of characters (see e.g. \cite[Lemma~3.2]{pollack03}). Hence, if there exists a fixed $r$ such that $P_{n,r}(\eta^\pm,z_{n,-r})=0$ for all $n$, then $\col^\pm(\mathbf{z})=0$. Therefore, we have
\[\ker(\col^\pm)=\Tw_r\left(\Hpm\left(T_{\bar{f}}(k-1-r)\right)\right)\]
and we are done.\end{proof}

\begin{corollary}
We have $\ker\LL_{\eta^\pm}=\ker\left(\col^\pm\right)=\Tw_r\left(\Hpm\left(T_{\bar{f}}(k-1-r)\right)\right)$ for any integer $0\le r\le k-2$.
\end{corollary}


\subsection{Pontryagin duality}\label{PD}

We have seen that $\ker(\col^\pm)$ can be written in terms of $H^1_\pm$, about which we now say a little bit more. The Pontryagin duality gives a pairing:
\begin{equation}\label{pair}
H^1(\Qpn,V_f/T_f(r+1))\times H^1(\Qpn,T_{\bar{f}}(k-1-r))\rightarrow E/\OO_E.
\end{equation}
We can describe the annihilator of $H^1_\pm(\Qpn,T_{\bar{f}}(k-1-r))$ under this pairing explicitly:

\begin{lemma}\label{ann}
$H^1_f(\Qpn,T_f(r+1))^\pm{\otimes}E/\OO_E\hookrightarrow H^1(\Qpn,V_f/T_f(r+1))$ and it can be identified as the annihilator of $H^1_\pm(\Qpn,T_{\bar{f}}(k-1-r))$ under (\ref{pair}).
\end{lemma}

\begin{proof}By definitions, we have an exact sequence
\[
0\rightarrow H^1_\pm(\Qpn,T_{\bar{f}}(k-1-r))\rightarrow H^1(\Qpn,T_{\bar{f}}(k-1-r))\rightarrow\mathrm{Hom}(H^1_f(\Qpn,T_f(r+1))^\pm,\OO_E).
\]
Taking Pontryagin duals, we have
\[
H^1_f(\Qpn,T_f(r+1))^\pm{\otimes}E/\OO_E\rightarrow H^1(\Qpn,V_f/T_f(r+1))\rightarrow H^1_\pm(\Qpn,T_{\bar{f}}(k-1-r))^{\vee}\rightarrow0.
\]
Therefore, the second part of the lemma follows from the first. Recall that $(V_f/T_f(r+1))^{G_{\Qpn}}=0$ by Lemma~\ref{inv}, so we have
\[
H^1_f(\Qpn,T_f(r+1)){\otimes}E/\OO_E\hookrightarrow H^1_f(\Qpn,V_f/T_f(r+1))\subset H^1(\Qpn,V_f/T_f(r+1)).
\]
Hence, it suffices to show that we have inclusion 
\[
H^1_f(\Qpn,T_f(r+1))^\pm{\otimes}E/\OO_E\hookrightarrow H^1_f(\Qpn,T_f(r+1)){\otimes}E/\OO_E.
\]
 But this follows from \cite[Lemma~8.17]{kobayashi03}.
\end{proof}

We write $H^1_f(\Qpn,V_f/T_f(j))^\pm$ for $H^1_f(\Qpn,T_f(j))^\pm{\otimes}E/\OO_E$, which is identified as a subgroup of $H^1_f(\Qpn,V_f/T_f(j))$. Note that it corresponds to the definition of $E^\pm(\QQ_{p,n})\otimes\Qp/\ZZ_p$ given in \cite{kobayashi03} and this is used to define $\Sel_p^\pm$ in Section~\ref{selmer}.


\section{Images of the Coleman maps}
\label{images}

In this section, we describe the images of $\col^\pm$. By Corollary~\ref{proj}, any elements of $H^1(\QQ_{p,n},T_{\bar{f}}(k-1))$ can be lifted to a global element of $\HIw(T_{\bar{f}}(k-1))$. Hence, we can in fact think of $\LL_{\eta^\pm,n}$ and $\col_n^\pm$ as maps from $H^1(\Qpn,T_{\bar{f}}(k-1))$ to $E[G_n]$. This allows us to give a description of $\image(\col^\pm)$ by studying $\image(\col^\pm_n)$.

In \cite[Section 8]{kobayashi03}, the images of the plus and minus Coleman maps for elliptic curves over $\QQ$ are shown to be the following:
\begin{eqnarray*}
{\rm Im(Col^+)}&=&(\gamma-1)\Lambda_{\calO_E}(G_\infty)+\left(\sum_{\sigma\in\Delta}\sigma\right)\Lambda_{\calO_E}(G_\infty),\\
{\rm Im(Col^-)}&=&\Lambda_{\calO_E}(G_\infty).
\end{eqnarray*}
In particular, the $\Delta$-invariant part of $\image(\col^\pm)$ is the whole of $(\sum_{\sigma\in\Delta}\sigma)\Lambda_{\calO_E}(G_\infty)$ (which we identify with $\Lambda_{\calO_E}(\Gamma)$). For a general $f$, we unfortunately do not know whether the images of the Coleman maps are inside $\Lambda_{\calO_E}(G_\infty)$ or not. However, after multiplying by a power of $\varpi$, we show that the $\Delta$-invariant part of $\image(\col^\pm)$ agree with the above descriptions and the same can be said for the whole of $\image(\col^-)$.


\subsection{Divisibility by $\Phi_m(\gamma)$}\label{image}
We have seen that the image of $\LL_{\eta^\pm}$ is divisible by $\log_{p,k}^\pm$. We give a necessary and sufficient condition for such divisibility at the finite level below.

Recall that $G_\infty=\Gal(k_\infty/\QQ)\cong\Delta\times\Gamma$ where $\Delta$ is a finite group of order $p-1$, $\Gamma\cong\ZZ_p$ and $\gamma$ is a fixed topological generator of $\Gamma$. We have 
\[\OO_E[G_n]\cong\OO_E[\Delta][\gamma]/(\gamma^{p^{n-1}}-1)\quad\text{and}\quad\Phi_m(\gamma)=1+\gamma^{p^{m-1}}+\gamma^{2p^{m-1}}+\cdots+\gamma^{(p-1)p^{m-1}}.\]
 Therefore, if $m\ge n$, then $\Phi_m(\gamma)=p$ in $\OO_E[G_n]$, so we only consider $m<n$ here.

\begin{lemma}\label{lin}
Let $m<n$ and 
\[ 
f=\sum_{\underset{\sigma\in\Delta}{r\ {\rm{mod}}\ p^{n-1}}}c_{r,\sigma}\cdot\sigma\cdot \gamma^r\in\OO_E[G_n].
\]
 For each $\sigma\in\Delta$ and $r\mod p^m$, write 
\[
b_{r,\sigma}=c_{r,\sigma}+c_{r+p^{m},\sigma}+\cdots+c_{r-p^{m},\sigma}.
\]
 Then, $f$ is divisible by $\Phi_m(\gamma)$ in $\calO_E[G_n]$ if and only if $b_{r,\sigma}=b_{s,\sigma}$ whenever $r\equiv s\mod p^{m-1}$.
\end{lemma}

\begin{proof}
Let $f=g\Phi_m(\gamma)$ and $g=\sum a_{r,\sigma}\cdot\sigma\cdot \gamma^r\in\OO_E[G_n]$. Then the coefficient of $\sigma \gamma^r$ in $f$ is
\[a_{r,\sigma}+a_{r-p^{m-1},\sigma}+\cdots+a_{r-(p-1)p^{m-1},\sigma}.\]
 Hence, $b_{r,\sigma}$ as defined in the statement of the lemma is just the sum of the coefficients $a_{s,\sigma}$ of $g$ with $s\equiv r\mod p^{m-1}$. Hence $b_{r,\sigma}=b_{s,\sigma}$ whenever $r\equiv s\mod p^{m-1}$.

Conversely, let $\sum c_{r,\sigma}\cdot\sigma\cdot \gamma^r\in\OO_E[G_n]$ and define $b_{r,\sigma}$ as in the statement of the lemma. Assume that $b_{r,\sigma}=b_{s,\sigma}$ for all $r\equiv s\mod p^{m-1}$. Let $f_\sigma(\gamma)=\sum_r c_{r,\sigma}\cdot \gamma^r$, so $f=\sum f_\sigma\cdot\sigma$. We have
\begin{eqnarray*}
f_\sigma(\zeta_{p^m})&=&\sum_{r\ \textrm{mod}\ p^m}\left(\sum_{s\equiv r(p^m)}c_{s,\sigma}\right)\zeta_{p^m}^r\\
&=&\sum_{r\ \textrm{mod}\ p^m}b_{r,\sigma}\zeta_{p^m}^r\\
&=&\sum_{s\ \textrm{mod}\ p^{m-1}}b_{s,\sigma}\sum_{r\equiv s(p^{m-1})}\zeta_{p^m}^r\\
&=&0.
\end{eqnarray*}
Hence, $\Phi_m(\gamma)$ divides $f$ and we are done.
\end{proof}

Applying this to the image of $\LL_{\eta^\pm,n}$, we have:
\begin{corollary}\label{lin2}
For any $z\in H^1(\Qpn,T_{\bar{f}}(k-1))$, $\LL_{\eta^\pm,n}(z)$ is divisible by $\Phi_m(\gamma)$ in $E[G_n]$ if $m\in S_n^\pm$.
\end{corollary}

\begin{proof}
The image of $\LL_{\eta^\pm,n}(z)$ is given by the following composition
\[
H^1(\Qpn,T_{\bar{f}}(k-1))\stackrel{\sim}{\longrightarrow}\textrm{Hom}_{\OO_E}\left(H^1(\Qpn,T_f(1)),\OO_E\right)\rightarrow E[G_n]
\]
where the first isomorphism is induced by the pairing (\ref{pairing}) and the second map is given by
\begin{equation}\label{defn}
\begin{split}
\textrm{Hom}_{\OO_E}\left(H^1(\Qpn,T_f(1)),\OO_E\right)&\rightarrow E[G_n]\\
\theta&\mapsto\sum_{\tau\in G_n}\theta(\exp_{n,1}(\gamma_{n,1}(\eta_1^{\pm})^\tau)\tau,
\end{split}
\end{equation}
with $\theta$ extended to an element of $\textrm{Hom}_{E}(H^1(\Qpn,V_f(1)),E)$ in the natural way. Hence, it is enough to show that the coefficients $\theta(\exp_{n,1}(\gamma_{n,1}(\eta_1^{\pm})^\tau)$, as $\tau\in G_n$ varies, satisfy the relations described in Lemma~\ref{lin}. Recall that $\exp_{n,1}$ gives an isomorphism 
\[
\Qpn\otimes\DD(V_f(1))/E\cdot\omega_1\rightarrow H^1_f(\Qpn,V_f(1)).
\]
Therefore, it is in fact enough to show that $\gamma_{n,1}(\eta^{\pm}_1)^\tau\mod\omega$ satisfy the relations in Lemma~\ref{lin}. Let $\sigma\in\Delta$ and $r\in\ZZ/p^m\ZZ$. For $\eta=\eta^\pm$, we write
\begin{align*}
\eta_{r,\sigma}&=\sum_{s\equiv r(p^m)}\gamma_{n,1}(\eta_1)^{\sigma \gamma^s}\\
&=p^{-m-1}\left((1-\vp)^{-1}(\eta_1)+\zeta_p\otimes\vp^{-1}(\eta_1)+\cdots+\zeta_{p^{m+1}}\otimes\vp^{-m-1}(\eta_1)\right)^{\sigma \gamma^r}.
\end{align*} 
Therefore, if $\vp^{-m-1}(\eta_1)\equiv0\mod\omega$, then $\eta_{r,\sigma}=\eta_{s,\sigma}$ for $r\equiv s\mod p^{m-1}$, as $(\zeta_{p^m})^{\sigma \gamma^r}=(\zeta_{p^m})^{\sigma \gamma^s}$. Hence, by the definitions of $\eta^\pm$ as given in the proof of Proposition~\ref{pm}, we are done.
\end{proof}

By considering its image modulo $(u^{-j}\gamma)^{p^{n-1}}-1$ similarly, one can deduce Proposition~\ref{pm}. We can in fact say a bit more about the image of $\LL_{\eta^+,n}$.
\begin{lemma}\label{lin3}
If $\LL_{\eta^+,n}(z)=\sum c_{r,\sigma}\cdot\sigma\cdot \gamma^r$, then $\sum_r c_{r,\sigma}$ is independent of $\sigma$.
\end{lemma}
\begin{proof}
For each  $\sigma\in\Delta$, we have
\[
\sum_r\gamma_{n,1}(\eta^+_1)^{\sigma \gamma^r}=p^{-1}\left((1-\vp)^{-1}(\eta_1^+)+\zeta_p\otimes\vp^{-1}(\eta_1^+)\right)^\sigma.
\]
But $\vp^{-1}(\eta_1^+)\equiv0 \mod \omega$, so we are done.
\end{proof}

We will see later on that these conditions in fact characterise the images of $\LL_{\eta^\pm,n}$ completely.


\subsection{Images of $\log_{p,k}^\pm$ in $\OO_E[G_n]$}
We now fix an integer $j$ such that $0<j\le k-2$.

\begin{lemma}\label{padic}
Let $x\in1+p\ZZ_p$. There exists a constant $c$ such that for any positive integer $n$, $v_p(x^{p^n}-1)=n+c$.
\end{lemma}

\begin{proof}
Let $x=1+m$ where $m\in p\ZZ_p$, so $v_p(m)\ge1$. We have expansion
\[
x^{p^n}-1=(1+m)^{p^n}-1=m^{p^n}+\binom{p^n}{p^n-1}m^{p^n-1}+\cdots+\binom{p^n}{1}m.
\]
For $r>0$, $v_p(\binom{p^n}{r})=n-v_p(r)$, so 
\[
v_p\left(\binom{p^n}{r}m^r\right)=rv_p(m)-v_p(r)+n.
\]
If $r=p^sa$ where $p\nmid a$ and $a>1$, then 
\[
v_p\left(\binom{p^n}{r}m^r\right)>v_p\left(\binom{p^n}{p^s}m^{p^s}\right).
\]
 Therefore, the set $\left\{v_p\left(\binom{p^n}{r}m^r\right):r>0\right\}$ takes its minimum value at $r=p^s$ for some $s$.

Consider the curve $f(t)=p^tv_p(m)-t$, for $t\in\mathbb{R}$. It has a unique global minimum when $p^t=(v_p(m)\log p)^{-1}$, so the curve is strictly increasing on $t\ge0$. Therefore, for a fixed $n$, the minimum of the values 
\[
v_p\left(\binom{p^n}{p^s}m^{p^s}\right)=p^sv_p(m)-s+n
\]
is just $v_p(m)+n$, which is attained at a unique $s$, hence the result.
\end{proof}

\begin{corollary}
If $m\ge n$, then $\Phi_m(u^{-j}\gamma)/p$ is congruent to a unit of $\Zp$ modulo $\gamma^{p^{n-1}}-1$.
\end{corollary}

\begin{proof}
By definition,
\[
\Phi_m(u^{-j}\gamma)=\frac{(u^{-j}\gamma)^{p^m}-1}{(u^{-j}\gamma)^{p^{m-1}}-1},
\]
so as elements of $\calO_E[G_n]$, we have
\[
\frac{1}{p}\Phi_m(u^{-j}\gamma)=\frac{u^{-jp^m}-1}{p(u^{-jp^{m-1}}-1)}.
\]
But $u\in1+p\ZZ_p$ by definition, so we are done by Lemma~\ref{padic}.
\end{proof}

\begin{remark}\label{omega}We have $\displaystyle\log_{p,k}^\pm\equiv p^{1-k}\lambda_\pm\prod_{j=0}^{k-2}\omega_n^\pm(u^{-j}\gamma)\mod(\gamma^{p^{n-1}}-1)$ where $\lambda_\pm$ is a unit of $\Zp$ and $\omega_n^\pm$ is defined by
\begin{eqnarray*}
\omega_n^+(1+X)&=&\prod_{1\le m<n/2}\Phi_{2m}(1+X)/p,\\
\omega_n^-(1+X)&=&\prod_{1\le m<(n+1)/2}\Phi_{2m-1}(1+X)/p.
\end{eqnarray*}\end{remark}


\subsection{The images of $\mathrm{Col}_n^\pm$}

Let $R_{n,j}^\pm$ be the $E$-vector spaces defined by \eqref{defineevenodd}. We have:

\begin{lemma}\label{Dim}
The dimensions of the $E$-vector spaces $R_{n,j}^\pm$ are given by
\begin{eqnarray*}
\dim_ER_{n,j}^+&=&1+\sum_{1\le m\le n/2}p^{2m-2}(p-1)^2\\
\dim_ER_{n,j}^-&=&p-1+\sum_{1\le m\le (n-1)/2}p^{2m-1}(p-1)^2
\end{eqnarray*}
\end{lemma}

\begin{proof}
By \eqref{evenodd}, we have
\begin{eqnarray*}
\dim_ER_{n,j}^+&=&\dim_{\Qp}\Qp+\sum_{1\le m\le n/2}\dim_{\Qp}\Qp^{(2m)},\\
\dim_ER_{n,j}^-&=&\dim_{\Qp}\Qp+\sum_{1\le m\le (n-1)/2}\dim_{\Qp}\Qp^{(2m+1)}.
\end{eqnarray*}
For $m>1$, \eqref{Qpndeco} implies that
\begin{eqnarray*}
\dim_{\Qp}\Qp^{(m)}&=&\dim_{\Qp}\QQ_{p,m}-\dim_{\Qp}\QQ_{p,m-1}\\
&=&p^{m-1}(p-1)-p^{m-2}(p-1)\\
&=&p^{m-2}(p-1)^2
\end{eqnarray*}
and $\dim_{\Qp}\Qp^{(1)}=p-2$, so we are done.
\end{proof}

The dimensions of these vector spaces enables us to obtain the following:

\begin{proposition}\label{finite}
Let $\displaystyle f=\sum_{\sigma\in\Delta}\sum_{r=0}^{p^{n-1}-1} a_{r,\sigma}\cdot\sigma\cdot u^r\in E[G_n]$. If $\omega_n^\pm$ is as defined in Remark~\ref{omega}, then:
\begin{itemize}
\item[(a)]There exists $z\in H^1(\Qpn,V_{\bar{f}}(k-1))$ such that $\col^-_n(z)\equiv f\mod \omega_n^+(\gamma)$.
\item[(b)]If moreover $\displaystyle\sum_{r}a_{r,\sigma_1}=\sum_r a_{r,\sigma_2}$ for all $\sigma_1,\sigma_2\in\Delta$, then there exists $z\in H^1(\Qpn,V_{\bar{f}}(k-1))$ such that $\col^+_n(z)\equiv f\mod \omega_n^-(\gamma)$.
\end{itemize}
\end{proposition}

\begin{proof}
We only prove (b), as (a) can be proved in the same way. Define
\[
U_n=\left\{g=\sum c_{r,\sigma}\cdot\sigma\cdot \gamma^r\in E[G_n]:\log_{p,k}^+|g,\sum_{r}c_{r,\sigma_1}=\sum_{r}c_{r,\sigma_1}\forall\sigma_1,\sigma_2\in\Delta\right\}.
\]
Then $U_n$ is a vector subspace of $E[G_n]$ over $E$. By remark \ref{omega}, 
\[
\log_{p,k}^+\equiv p^{1-k}\lambda_+\prod_{j=0}^{k-2}\omega_n^+(u^{-j}\gamma)\mod(\gamma^{p^{n-1}}-1)
\]
 for some $\lambda_+\in\OO_E^\times$. Since $\omega_n^+(u^{-j}(1+X))$ and $(1+X)^{p^{n-1}}-1$ are coprime for $j>0$, $\log_{p,k}^+|g$ if and only if $\omega_n^+(\gamma)|g$. But $\Phi_{m_1}$ and $\Phi_{m_2}$ are coprime if $m_1\ne m_2$, so $\omega_n^+(\gamma)|g$ if and only if $\Phi_m(\gamma)|g$ for all even $m<n$.

Let $g=\sum c_{r,\sigma}\cdot\sigma\cdot u^r$. For each even $m<n$, let 
\[
b_{r,\sigma}^{(m)}=c_{r,\sigma}+c_{r+p^{m},\sigma}+\cdots+c_{r-p^{m},\sigma}.
\]
 Then, by Lemma~\ref{lin}, $\Phi_m(\gamma)|g$ if and only if $b_{r,\sigma}^{(m)}=b_{s,\sigma}^{(m)}$ for all $\sigma\in\Delta$ and $r\equiv s\mod p^{m-1}$. For each such $m$ and $\sigma\in\Delta$, there are $p^{m-1}$ values of modulo $p^{m-1}$, each is equated to $p-1$ different values. Since $|\Delta|=p-1$, there are $p^{m-1}(p-1)^2$ linearly independent equations for each $m$. Together with the equations of $\sum_{r}c_{r,\sigma}$, there are in total 
\[ 
p-2+\sum_{1\le m\le n/2}p^{2m-1}(p-1)^2
\]
 equations describing the coefficients of elements of the $U_n$, which gives the codimension of $U_n$ over $E$ in $E[G_n]$.

By Corollary~\ref{lin2} and Lemma~\ref{lin3}, for $z\in H^1(\Qpn,V_{\bar{f}}(k-1))$, $\LL_{\eta^+,n}(z)$ lies inside the above subspace. But the dimension of the image is given by $\dim_{E}R_{n,1}^+$ which is the same as the dimension of $U_n$ by Lemma~\ref{Dim}, so $\LL_{\eta^+,n}\left(H^1(\Qpn,V_{\bar{f}}(k-1))\right)=U_n$ as $E$-vector spaces and there exists some $z$ such that $\LL_{\eta^+,n}(z)=g.$ This implies
\[
\log_{p,k}^+\col^+_n(z)\equiv f\log_{p,k}^+\mod(\gamma^{p^{n-1}}-1).
\]
 The factors of $\omega_n^+(u^{-j}\gamma)$ on both sides can be cancelled out for $j>0$ as $\omega_n^+(u^{-j}\gamma)$ is coprime to $\omega_n^+(\gamma)$. Since $p^{n-1}(\gamma-1)\omega_n^+(\gamma)\omega_n^-(\gamma)=\gamma^{p^{n-1}}-1$, we deduce that $\col^+_n(z)\equiv f\mod((\gamma-1)\omega_n^-(\gamma))$, which implies (b).
\end{proof}


\subsection{The images of $\col^\pm$}
In the previous section, we studied the images of $H^1(\Qpn,V_{\bar{f}}(k-1))$ under $\col^\pm_n$. To understand the images of $\col^\pm$, we have to understand those of $H^1(\Qpn,T_{\bar{f}}(k-1))$ as well.

\begin{lemma}\label{inte}
For all $n$, there exist $r_n^\pm\in\ZZ$ such that 
\[
\LL_{\eta^\pm,n}(H^1(\Qpn,T_{\bar{f}}(k-1)))=\LL_{\eta^\pm,n}(H^1(\Qpn,V_{\bar{f}}(k-1)))\cap \varpi^{r_n^\pm}\OO_E[G_n].
\]
\end{lemma}

\begin{proof}
Note that $\exp_{n,1}(\gamma_{n,1}(\eta^{\pm}_1))\ne0$. As an element of $H^1(\Qpn,T_f(1))$, it lifts to a cocycle on $G_{\Qpn}$. By considering the image of this cocycle in $V_f(1)$, which is invariant under the action of $G_n$, there exists $r_n^\pm$ such that
\[
\varpi^{-r_n^\pm}\exp_{n,1}(\gamma_{n,1}(\eta^{\pm})^\tau)\in H^1(\Qpn,T_f(1))\setminus \varpi H^1(\Qpn,T_f(1))
\]
for all $\tau\in G_n$.

Recall from (\ref{defn}) that $\LL_{\eta^\pm,n}$ is given by:
\begin{equation*}
\begin{split}
\Hom_{E}\left(H^1(\Qpn,V_f(1)),E\right)&\rightarrow E[G_n]\\
\theta&\mapsto\sum_{\tau\in G_n}\theta(\exp_{n,1}(\gamma_{n,1}(\eta^{\pm}_1)^\tau)\tau,
\end{split}
\end{equation*}
where we have identified $\Hom_{E}\left(H^1(\Qpn,V_f(1)),E\right)$ with $H^1(\Qpn,V_{\bar{f}}(k-1))$. Under this identification, $H^1(\Qpn,T_{\bar{f}}(k-1))$ corresponds to the set of maps which send $H^1(\Qpn,T_f(1))$ (which is identified as a subset of $H^1(\Qpn,V_f(1))$ as discussed in Section~\ref{kernel}) to $\OO_E$. Therefore, we have
\[
\left\{\theta(\exp_{n,1}(\gamma_{n,1}(\eta^{\pm}_1)^\tau):\theta\in H^1(\Qpn,T_{\bar{f}}(k-1))\right\}=\varpi^{r_n^\pm}\OO_E
\]
for all $\tau\in G_n$. This implies that the LHS of the equation in the statement of the lemma is contained in the RHS.

Conversely, if $x$ is an element of the RHS of the equation, there exists $\theta\in H^1(\Qpn,V_{\bar{f}}(k-1))$ such that $\sum_{\tau\in G_n}\theta(\exp_{n,1}(\gamma_{n,1}(\eta^{\pm}_1)^\tau)\tau=x$ by Proposition~\ref{finite}. In particular,
\[
\theta\left(\varpi^{-r_n^\pm}\exp_{n,1}(\gamma_{n,1}(\eta_1^{\pm})^\tau\right)\in\OO_E
\]
for all $\tau\in G_n$. Hence, there exists $\tilde{\theta}\in H^1(\Qpn,T_{\bar{f}}(k-1))$ which agree with $\theta$ on the set $\{\varpi^{-r_n^\pm}\exp_{n,1}(\gamma_{n,1}(\eta^{\pm}_1)^\tau):\tau\in G_n\}$, so $x\in$LHS. 
\end{proof}

\begin{lemma}\label{rn}
Let $r_n^\pm$ be the integers defined in Lemma~\ref{inte}, then there exist $c_\pm$ such that $r_n^\pm=-e(k-1)\lfloor n/2\rfloor+c_{\pm}$ for $n$ sufficiently large where $e$ is the ramification degree of $E$.
\end{lemma}

\begin{proof}
By Remark~\ref{log}, 
\[\Omega_{V_f(1),1}((1+X)\otimes\eta^\pm_1)=O(\log_p^{(k-1)/2}),\]
which implies that the $n$th component of $\Omega_{V_f(1),1}((1+X)\otimes\eta^\pm_1)$, which is $\exp_{n,1}\left(\gamma_{n,1}(\eta^\pm_1)\right)$ satisfies 
\[
\exp_{n,1}\left(\gamma_{n,1}(\eta^\pm_1)\right)\in \varpi^{-e(k-1)\lfloor n/2\rfloor+c_{\pm}}H^1(\Qpn,T_{f}(1))
\]
for some constant $c_\pm$ independent of $n$. 

Recall that $\HIw(T_{f}(1))$ is free of rank 2 over $\Lambda_{\calO_E}(G_\infty)$. Fix a basis $z_1,z_2$, say. Note that $(1+X)\otimes\eta^\pm_1$ form a $\Lambda_{E}(G_\infty)$-basis for $\DD_\infty(V_f)$. The determinant of 
\[
\Omega_{V_f(1),1}:\mathcal{H}_{\infty}(G_\infty)\otimes\DD_{\infty}(V_f(1))\rightarrow\mathcal{H}_{\infty}(G_\infty){\otimes}\HIw(T_f(1))
\] 
with respect to these bases, as a $\mathcal{H}_\infty(G_{\infty})$-homomorphism, is given by
\[
\prod_{j=0}^{k-2}\log_p(u^j\gamma)\sim\log_{p}^{k-1}
\]
up to a unit of $\Lambda_E(G_\infty)$ (this is the $\delta(V)$-conjecture of \cite{perrinriou94}, which can be deduced from the explicit reciprocity law of Colmez \cite{colmez98}). But Theorem~\ref{pollack} says that $\log_{p,k}^\pm\sim\log_p^{(k-1)/2}$. Hence, we in fact have 
\[\Omega_{V_f(1),1}((1+X)\otimes\eta^\pm)\sim\log_p^{(k-1)/2}.\]
Therefore, we can choose $c_\pm$ such that
\[
\exp_{n,1}\left(\gamma_{n,1}(\eta^\pm_1)\right)\notin \varpi^{-e(k-1)\lfloor n/2\rfloor+c_{\pm}+1}H^1(\Qpn,T_{f}(1)),
\]
so $r_n^\pm=-e(k-1)\lfloor n/2\rfloor+c_{\pm}$, for $n$ sufficiently large.
\end{proof}

On combining these two lemmas, we have:

\begin{corollary} \label{trivialchar}\label{inteim}
If $\theta$ is the trivial character on $\Delta$, then there exist $s^\pm$ such that
\[\col^\pm\left(\HIw(T_{\bar{f}}(k-1))\right)^\theta=\varpi^{s_{\pm}}\Lambda_{\calO_E}(\Gamma).\]
\end{corollary}

\begin{proof}
By Proposition~\ref{finite} and Lemma~\ref{inte}, for sufficiently large $n$,
\[
\varpi^{r_n^\pm}\left(\sum_{\sigma\in\Delta}\sigma\right)\prod_{j=0}^{k-2}\tilde{\omega}_n^\pm(u^{-j}\gamma) \in\LL_{\eta^\pm,n}\left(H^1(\Qpn,T_{\bar{f}}(k-1))\right)
\] 
where
\begin{eqnarray*}
\tilde{\omega}_n^+(1+X)&=&\prod_{1\le m<n/2}\Phi_{2m}(1+X),\\
\tilde{\omega}_n^-(1+X)&=&\prod_{1\le m<(n+1)/2}\Phi_{2m-1}(1+X).
\end{eqnarray*}
Hence, by Remark~\ref{omega} and Lemma~\ref{rn}, there exist constants $s^\pm$ (independent of $n$) such that
\[
\varpi^{s^\pm}\left(\sum_{\sigma\in\Delta}\sigma\right)\log_{p,k}^\pm \in\LL_{\eta^\pm,n}\left(H^1(\Qpn,T_{\bar{f}}(k-1))\right)
\]
and 
\[
\LL_{\eta^\pm,n}\left(H^1(\Qpn,T_{\bar{f}}(k-1))\right)\subset \varpi^{s^\pm}\log_{p,k}^\pm\calO_E[G_n].
\]
But $\log_{p,k}^\pm\col^\pm=\LL_{\eta^\pm}$, so we have
\[
\varpi^{s^\pm}\sum_{\sigma\in\Delta}\sigma\in\col^\pm\left(H^1(\Qpn,T_{\bar{f}}(k-1)\right)\mod\tilde{\omega}_n^\mp(\gamma).
\]
Therefore, we are done since 
\[\lim_{\leftarrow}\Lambda_{\calO_E}(G_\infty)/\tilde{\omega}_n^\pm(\gamma)=\Lambda_{\calO_E}(G_\infty)\quad\text{and}\quad\Lambda_{\calO_E}(G_\infty)^\theta=\left(\sum_{\sigma\in\Delta}\sigma\right)\Lambda_{\calO_E}(G_\infty).
\]
\end{proof}

\begin{remark}\label{generalchar}
It is clear that we can replace $\theta$ by an arbitrary character on $\Delta$ for the minus map in the corollary.
\end{remark}


\section{$\pm$-Selmer groups}\label{selmer}

Throughout this section, with the exception of Sections~\ref{itiscotorsion} and \ref{myconj}, assumptions (1) and (2) is not necessary.

Let $f$ be a modular form as in Section~\ref{modularforms}, $K$ a number field, the $p$-Selmer groups of $f$ over $K$ are defined by the following:  
\begin{eqnarray*}
\Sel^0_p(f/K)&=&\ker\left(H^1(K,V_f/T_f(1))\rightarrow\prod_v H^1(K_v,V_f/T_f(1))\right)\\
\Sel_p(f/K)&=&\ker\left(H^1(K,V_f/T_f(1))\rightarrow\prod_v\frac{H^1(K_v,V_f/T_f(1))}{H^1_f(K_v,V_f/T_f(1))}\right)
\end{eqnarray*}
where $v$ runs through the places of $K$.

We write $k_n$ for $\QQ$ adjoining all the $p^n$th roots of unity and $k_\infty=\cup k_n$. Since there is a unique place above $p$ in $k_n$, we write this place as $p$ as well. Note that the completion of $k_n$ at $p$ is isomorphic to $\Qpn$. For $f$ satisfying assumptions (1) and (2), let $H^1_f(\Qpn,V_f/T_f(1))^\pm$ be as defined in Section~\ref{PD}. For all $n\ge0$, we define the plus and minus Selmer groups by
\[
\Sel_p^\pm(f/k_n)=\ker\left(\Sel_p(f/k_n)\rightarrow\frac{H^1(\Qpn,V_f/T_f(1))}{H^1_f(\Qpn,V_f/T_f(1))^\pm}\right).
\]
In this section, we show that $\Sel_p(f/k_\infty)$ is not $\Lambda_{\calO_E}(G_\infty)$-cotorsion when $f$ is supersingular at $p$. When $f$ satisfies assumptions (1) and (2), we  show that $\Sel_p^\pm(f/k_\infty)=\displaystyle\lim_{\rightarrow}\Sel_p^\pm(f/k_n)$ is $\Lambda_{\calO_E}(G_\infty)$-cotorsion. 


\subsection{Restricted ramification}

We now describe the Selmer groups defined above using restricted ramification. Let $S$ be a finite set of places of a number field $K$ containing all infinite places, all primes above $p$ and those dividing $N$. Then, by \cite[Lemma~I.5.3]{rubin00},
\begin{equation}\label{restrict}
H^1(G_{S,K},V_f/T_f(1))=\ker\left(H^1(K,V_f/T_f(1))\rightarrow\prod_{v\notin S}\frac{H^1(K_v,V_f/T_f(1))}{H^1_f(K_v,V_f/T_f(1))}\right)
\end{equation}
where $G_{S,K}$ is the Galois group of the maximal extension of $K$ unramified outside $S$. Therefore, we can rewrite $\Sel_p$ as
\begin{equation}\label{sim}
\Sel_p(f/K)=\ker\left(H^1(G_{S,K},V_f/T_f(1))\rightarrow\bigoplus_{v\in S} \frac{H^1(K_v,V_f/T_f(1))}{H^1_f(K_v,V_f/T_f(1))}\right).
\end{equation}

If $f$ satisfies assumptions (1) and (2), we write $H^1_f(k_{n,v},V_f/T_f(1))^\pm=H^1_f(k_{n,v},V_f/T_f(1))$ for $v\nmid p$. Then,
\begin{equation}
\label{simpm}\Sel_p^\pm(f/k_n)=\ker\left(H^1(G_{S,k_n},V_f/T_f(1))\rightarrow\bigoplus_{v\in S}\frac{H^1(k_{n,v},V_f/T_f(1))}{H^1_f(k_{n,v},V_f/T_f(1))^\pm}\right).
\end{equation}

The next lemma enables us to give a similar alternative description of $\Sel_p^0$ as well.

\begin{lemma}\label{trivial}
With notation above, we have $H^1_f(K_v,V_f/T_f(1))=0$ for $v\nmid pN$.
\end{lemma}

\begin{proof}

If $v$ is an infinite place, we in fact have $H^1(K_v,V_f/T_f(1))=0$ as $p$ is odd (see e.g. \cite[Section~I.3.7]{rubin00}).

We now assume that $v$ is a finite place not dividing $pN$. Since $v\nmid p$,
\[
H^1_f(K_v,V_f(1))=H^1_\text{ur}(K_v,V_f(1))
\]
by definition and $H^1_f(K_v,V_f/T_f(1))$ is defined to be the image of $H^1_{\text{ur}}(K_v,V_f(1))$ in $H^1(K_v,V_f/T_f(1))$ under the natural map 
$H^1(K_v,V_f(1))\rightarrow H^1(K_v,V_f/T_f(1))$. By \cite[Section~I.3.2]{rubin00}, 
\[
H^1_\text{ur}(K_v,V_f(1))\cong V_f(1)^I/(\Fr-1)V_f(1)^I
\]
where $I$ is the inertia group of $K_v$ and  $\Fr$ is the Frobenius of $K_v^\text{ur}/K_v$. Hence, it suffices to show that $1$ is not an eigenvalue of $\Fr$. But $v$ is a good prime (i.e. $v\nmid N$), so the eigenvalues have absolute value $q_v^{(k-1)/2}$ where $q_v$ is the rational prime lying below $v$. Hence we are done.
\end{proof}

If $S$ is as above, Lemma~\ref{trivial} and (\ref{restrict}) implies that
\[
H^1(G_{S,K},V_f/T_f(1))=\ker\left(H^1(K,V_f/T_f(1))\rightarrow\prod_{v\notin S}H^1(K_v,V_f/T_f(1))\right).
\]
Therefore, by the definition of $\Sel_p^0$, we have:
\begin{equation}\label{sel0}
\Sel_p^0(f/K)=\ker\left(H^1(G_{S,K},V_f/T_f(1))\rightarrow\bigoplus_{v\in S}H^1(K_v,V_f/T_f(1))\right).
\end{equation}
As stated in the proof of Lemma~\ref{trivial}, $H^1(K_v,V_f/T_f(1))=0$ if $v$ is an infinite place. We can therefore simplify \eqref{sel0} further:
\begin{equation}\label{sel000}
\Sel_p^0(f/K)=\ker\left(H^1(G_{S,K},V_f/T_f(1))\rightarrow\bigoplus_{v\in S_f}H^1(K_v,V_f/T_f(1))\right).
\end{equation}
where $S_f$ denotes the set of finite places in $S$.


\subsection{Poitou-Tate exact sequences}
We now briefly review results on Poitou-Tate exact sequences. Details can be found in \cite[Section~A.3]{perrinriou95}.

With the above notation, let $S$ be a finite set of places of $K$ containing those above $p$ and the infinite places, then we have an exact sequence
\begin{equation}\label{pt1}
\bigoplus_{v\in S_f}H^0(K_v,V_f/T_f(1))\rightarrow H^2(G_{S,K},T_{\bar{f}}(k-1))^\vee\rightarrow H^1(G_{S,K},V_f/T_f(1))\rightarrow\bigoplus_{v\in S_f}H^1(K_v,V_f/T_f(1))
\end{equation}
where $S_f$ is again the set of finite places in $S$. On combining \eqref{pt1} and \eqref{sel000}, we have
\[
\bigoplus_{v\in S_f}H^0(K_v,V_f/T_f(1))\rightarrow H^2(G_{S,K},T_{\bar{f}}(k-1))^\vee\rightarrow\Sel_p^0(f/K).
\]
By taking duals and the fact that $H^0(K_v,V_f/T_f(1))^\vee=H^2(K_v,T_{\bar{f}}(k-1))$, we obtain
\begin{equation}\label{sel00}
\Sel_p^0(f/K)^\vee=\ker\left(H^2(G_{S,K},T_{\bar{f}}(k-1))\rightarrow\bigoplus_{v\in S_f}H^2(K_v,T_{\bar{f}}(k-1))\right)
\end{equation}

For each $v\in S_f$, let $A_v\subset H^1(K_v,T_{\bar{f}}(k-1))$ and $B_v\subset H^1(K_v,V_f/T_f(1))$
be $\calO_E$-modules so that they are orthogonal complements to each other under the Pontryagin duality. Define
\[
H^1_B(K,V_f/T_f(1))=\ker\left(H^1(G_{S,K},V_f/T_f(1))\rightarrow\bigoplus_{v\in S_f}\frac{H^1(K_v,V_f/T_f(1))}{B_v}\right).
\]
Then \cite[Proposition~A.3.2]{perrinriou95} says that we have an exact sequence
\begin{equation}\label{pt2}
\begin{split}
H^1(G_{S,K},T_{\bar{f}}(k-1))\rightarrow\bigoplus_{v\in S_f}\frac{H^1(K_v,T_{\bar{f}}(k-1))}{A_v}\rightarrow H^1_B(K,V_f/T_f(1))^\vee\\\rightarrow H^2(G_{S,K},T_{\bar{f}}(k-1))\rightarrow\bigoplus_{v\in S_f}H^2(K_v,T_{\bar{f}}(k-1)).
\end{split}
\end{equation}
Hence, we can combine (\ref{sel00}) and (\ref{pt2}) to obtain the following exact sequence:
\begin{equation}\label{pt3}
H^1(G_{S,K},T_{\bar{f}}(k-1))\rightarrow\bigoplus_{v\in S_f}\frac{H^1(K_v,T_{\bar{f}}(k-1))}{A_v}\rightarrow H^1_B(K,V_f/T_f(1))^\vee \rightarrow\Sel_p^0(f/K)^\vee\rightarrow0.
\end{equation}


\subsection{Cotorsionness}

\subsubsection{$\Sel_p(f/k_\infty)$ is not $\Lambda_{\calO_E}(G_\infty)$-cotorsion}\label{nottorsion}
We now prove our claim about $\Sel_p(f/k_\infty)^\vee$ in the introduction. Let $K=k_n$.  Take $B_v=H^1_f(k_{n,v},V_f/T_f(1))$ for $v\in S_f$ in (\ref{pt3}), then $A_v=H^1_f(k_{n,v},T_{\bar{f}}(k-1))$ by \cite[Proposition~3.8]{blochkato}. Hence, on combining \eqref{sim} and \eqref{pt3}, we have an exact sequence
\begin{equation}\label{prelimit}\begin{split}
H^1(G_{S,k_n},T_{\bar{f}}(k-1))\rightarrow\frac{H^1(\Qpn,T_{\bar{f}}(k-1))}{H^1_f(\Qpn,T_{\bar{f}}(k-1))}\oplus\bigoplus_{v|N}\frac{H^1(k_{n,v},T_{\bar{f}}(k-1))}{H^1_f(k_{n,v},T_{\bar{f}}(k-1))}\\\rightarrow\Sel_p(f/k_n)^\vee\rightarrow\Sel_p^0(f/k_n)^\vee\rightarrow0.\end{split}
\end{equation}

We are interested in taking inverse limit over $n$. For the terms coming from places dividing $N$, we can apply the following.

\begin{lemma}\label{vanish}
For each integer $n\ge0$, fix a prime $v(n)$ of $\Qpn$ not dividing $p$ such that $v(n+1)$ lies above $v(n)$, then 
\[\lim_{\underset{n,\cor}{\leftarrow}}\frac{H^1(k_{n,v(n)},T_{\bar{f}}(k-1))}{H^1_f(k_{n,v(n)},T_{\bar{f}}(k-1))}=0.\]
\end{lemma}
\begin{proof}The Pontryagin dual of the said inverse limit is $\displaystyle\lim_{\rightarrow}H^1_f(k_{n,v(n)},V_{f}/T_f(1))$, so the result follows immediately from Lemma~\ref{trivial} if $v(n)\nmid N$. The general case is proved in \cite[Section 17.10]{kato04} by considering $p$-cohomological dimensions.
\end{proof}

Therefore, on taking inverse limits in \eqref{prelimit}, we have the following exact sequence:
\begin{equation}\label{pt5}
\Hh^1_S(T_{\bar{f}}(k-1))\rightarrow\frac{\HIw(T_{\bar{f}}(k-1))}{\Hh_f(T_{\bar{f}}(k-1))}\rightarrow\Sel_p(f/k_\infty)^\vee\rightarrow\Sel_p^0(f/k_\infty)^\vee\rightarrow0
\end{equation}where $\displaystyle\Hh_f(\cdot)=\lim_{\stackrel{\longleftarrow}{n}}H^1_f(\QQ_{p,n},\cdot)$ and $\displaystyle\Hh^1_S(\cdot)=\lim_{\stackrel{\longleftarrow}{n}}H^1(G_{k_n,S},\cdot)\cong\Hh^1(\cdot)$ (see \cite[Proposition~7.1]{kobayashi03}).

\begin{proposition}
$\Sel_p(f/k_\infty)^\vee$ is not torsion over $\Lambda_{\calO_E}(G_\infty)$.
\end{proposition}
\begin{proof}
We consider the rank of each term appearing in \eqref{pt5}. By Theorem~\ref{katoglobal}, $\Hh^1_S(T_{\bar{f}}(k-1))$ is a torsion-free $\Lambda_{\calO_E}(G_\infty)$-module of rank 1. By \cite[Theorem~0.6]{perrinriou00b}, $\Hh_f(T_{\bar{f}}(k-1))=0$. By \cite[Proposition~3.2.1]{perrinriou94}, $\HIw(T_{\bar{f}}(k-1))$ is of rank 2 over $\Lambda_{\calO_E}(G_\infty)$. By \cite[proof of Proposition~7.1]{kobayashi03}, which is a purely algebraic proof and generalises to modular forms directly, $\Sel_p^0(f/k_\infty)^\vee$ is $\Lambda_{\calO_E}(G_\infty)$-torsion. Therefore, $\Sel_p(f/k_\infty)^\vee$ has $\Lambda_{\calO_E}(G_\infty)$-rank at least 1 and we are done.
\end{proof}


\subsubsection{$\Sel_p^\pm(f/k_\infty)$ is $\Lambda_{\calO_E}(G_\infty)$-cotorsion}\label{itiscotorsion}

We again set $K=k_n$. Let
\[
 B_v=
     \left\{
     \begin{array}{ll}
      H^1_f(k_{n,v},V_f/T_f(1))       & \text{if $v|N$}\\
     H^1(\Qpn,V_f/T_f(1))^\pm        & \text{if $v=p$.}  
     \end{array}\right.
\]
By \cite[Proposition~3.8]{blochkato} and Lemma~\ref{ann}, we have
\[
 A_v=
     \left\{
     \begin{array}{ll}
      H^1_f(k_{n,v},T_{\bar{f}}(k-1))       & \text{if $v|N$}\\
     H^1_\pm(\Qpn,T_{\bar{f}}(k-1))       & \text{if $v=p$.}  
     \end{array}\right.
\]
Hence, on combining (\ref{simpm}) with (\ref{pt3}), we obtain the following exact sequence:
\begin{equation}\label{pt6}
\begin{split}
H^1(G_{S,k_n},T_{\bar{f}}(k-1))\rightarrow\frac{H^1(\Qpn,T_{\bar{f}}(k-1))}{H^1_\pm(\Qpn,T_{\bar{f}}(k-1))}\oplus\bigoplus_{v|N}\frac{H^1(k_{n,v},T_{\bar{f}}(k-1))}{H^1_f(k_{n,v},T_{\bar{f}}(k-1))}\\\rightarrow\Sel_p^\pm(f/k_n)^\vee\rightarrow\Sel_p^0(f/k_n)^\vee\rightarrow0.
\end{split}
\end{equation}
Therefore, on taking inverse limits in \eqref{pt6} and applying Lemma~\ref{vanish}, we have the exact sequence
\begin{equation}\label{pt4}
\Hh^1_S(T_{\bar{f}}(k-1))\rightarrow\frac{\HIw(T_{\bar{f}}(k-1))}{\Hpm(T_{\bar{f}}(k-1))}\rightarrow\Sel_p^\pm(f/k_\infty)^\vee\rightarrow\Sel_p^0(f/k_\infty)^\vee\rightarrow0
\end{equation}
where $\Hpm(T_{\bar{f}}(k-1))$ is as defined in Section~\ref{kernel}, i.e. $\displaystyle\lim_{\leftarrow}H^1_\pm(\Qpn,T_{\bar{f}}(k-1))$.

\begin{proposition}
$\Sel_p^\pm(f/k_\infty)$ is $\Lambda_{\calO_E}(G_\infty)$-cotorsion.
\end{proposition}
\begin{proof}
Recall that $\ker(\col^\pm)=\Hpm(T_{\bar{f}}(k-1))$ from Section~\ref{kernel} and $\col^\pm(\kato)=L_{p}^\pm$ by \eqref{colemankato}. Therefore, the cokernel of the first map in \eqref{pt4} is killed by $L_p^\pm$. Therefore, if $L_p^\pm\ne0$, it would imply that the said cokernel is $\Lambda_{\calO_E}(G_\infty)$-torsion and the result would follow from the fact that $\Sel_p^0(f/k_\infty)^\vee$ is $\Lambda_{\calO_E}(G_\infty)$-torsion. Hence, we are done by the following lemma. 
\end{proof}

\begin{lemma}\label{notzero}
$L_p^\pm\ne0$.
\end{lemma}
\begin{proof}
The case when $f$ corresponds to an elliptic curve is proved in \cite[Corollary~5.11]{pollack03}. The general case can be proved similarly.

By \cite{pollack03}, if $\theta$ is a character on $G_n$ which does not factor through $G_{n-1}$ and $0\le r\le k-2$,
\begin{align*}
\chi^r\theta(L_{p}^+)=C_{n,r}^+(\theta)L(f,\theta,r+1)&\quad\text{if $n$ is even},\\
\chi^r\theta(L_{p}^-)=C_{n,r}^-(\theta)L(f,\theta,r+1)&\quad\text{if $n$ is odd}
\end{align*}
where $C_{n,r}^\pm(\theta)$ are nonzero constants. By \cite{rohrlich88}, $L(f,\theta,1)=0$ for finitely many $\theta$ if $k=2$. If $k\ge3$, $L(f,\theta,r+1)\ne0$ for $r+1\le(k-1)/2$ by \cite[Proposition~2]{shimura76}. Hence we are done.\end{proof}

\begin{corollary}\label{injection}
The first map in \eqref{pt4} is injective.
\end{corollary}

\begin{proof}
It follows from Theorem~\ref{katoglobal} and Lemma~\ref{notzero}.
\end{proof}

\begin{remark}
It is clear from the proof of Lemma~\ref{notzero} that $L_p^{\pm,\theta}\ne0$ for any character $\theta$ on $\Delta$. Therefore, $\Sel_p^\pm(f/k_\infty)^\theta$ is $\Lambda_{\calO_E}(\Gamma)$-cotorsion and we can associate to it a characteristic ideal, namely $\Char_{\Lambda_{\calO_E}(\Gamma)}\left(\Sel_p^\pm(f/k_\infty)^{\vee,\theta}\right)$.
\end{remark}

\subsection{Main conjectures}\label{myconj}
We now formulate a main conjecture and relate it to that of Kato. By Corollary~\ref{injection} and the fact that  $\Sel_p^0(f/k_\infty)^\vee\cong\Hh^2(T_{\bar{f}}(k-1))$ (see \cite{kurihara02}), we have an exact sequence
\[
0\rightarrow\Hh^1_S(T_{\bar{f}}(k-1))\rightarrow\image(\col^\pm)\rightarrow\Sel_p^\pm(f/k_\infty)^\vee\rightarrow \Hh^2(T_{\bar{f}}(k-1))\rightarrow0.
\]
If $\theta$ is a character on $\Delta$, then 
\[
\Char_{\Lambda_{\calO_E}(\Gamma)}(\Hh^1_S(T_{\bar{f}}(k-1))^\theta/\ZZ(T_{\bar{f}}(k-1))^\theta)=\Char_{\Lambda_{\calO_E}(\Gamma)}(\Hh^2(T_{\bar{f}}(k-1))^\theta)
\]
if and only if
\[
\Char_{\Lambda_{\calO_E}(\Gamma)}(\Sel_p^\pm(f/k_\infty)^{\vee,\theta})=\Char_{\Lambda_{\calO_E}(\Gamma)}(\image(\col^{\pm,\theta})/L_{p}^{\pm,\theta}).
\]
In other words, Kato's main conjecture (for $\bar{f}$) is equivalent to the following conjecture.
\begin{conjecture}\label{ourmc}$\Char_{\Lambda_{\calO_E}(\Gamma)}(\Sel_p^\pm(f/k_\infty)^{\vee,\theta})=\Char_{\Lambda_{\calO_E}(\Gamma)}(\image(\col^{\pm,\theta})/L_{p}^{\pm,\theta}).$
\end{conjecture}
Moreover, by Corollary \ref{inteim} and Remark~\ref{generalchar}, we have:
\begin{corollary}\label{ourMC}Let $\delta=\pm$. When $\theta=1$ or $\delta=-$, Conjecture~\ref{ourmc} is equivalent to
\[
\Char_{\Lambda_{\calO_E}(\Gamma)}\left(\Sel_p^\pm(f/k_\infty)^{\vee,\theta}\right)=\left(\varpi^{-s^\pm}L_{p}^{\pm,\theta}\right).
\]
\end{corollary}

\begin{remark}
It is clear that the RHS in Conjectures~\ref{ourmc} and \ref{ourMC} are contained in the LHS if the homomorphism $G_{\QQ}\rightarrow GL_{\calO_E}(T_{\bar{f}})$ is surjective or if we replace $\Lambda_{\calO_E}(\Gamma)$ by $\Lambda_E(G_\infty)$ by Theorem~\ref{katozetamc}.
\end{remark}


\section{CM forms}\label{CM}
We now follow the strategy of \cite{pollackrubin04} to prove that equality holds in Corollary~\ref{ourMC} (with $\theta=1$) for CM forms.

\subsection{Generality of CM forms}
We first briefly review the theory of CM modular forms. Details can be found in \cite[Section~15]{kato04}.

Let $K$ be an imaginary quadratic field with idele class group $C_K$. A Hecke character of $K$ is simply a continuous homomorphism $\phi:C_K\rightarrow\CC^\times$ with complex $L$-function
\[
L(\phi,s)=\prod_v(1-\phi(v)N(v)^{-s})^{-1}
\]
where the product runs through the finite places $v$ of $K$ at which $\phi$ is unramified, $\phi(v)$ is the image of the uniformiser of $K_{v}$ under $\phi$ and $N(v)$ is the norm of $v$.

Let $f$ be a modular form as defined in Section~\ref{modularforms} with complex multiplication, i.e. $L(f,s)=L(\phi,s)$ for some Hecke character $\phi$ of an imaginary quadratic field $K$. Then, for a good prime $p$,
\[
1-a_pp^{-s}+\epsilon(p)p^{k-1-2s}=\begin{cases}
1-\phi(p)p^{-2s}&\text{if $p$ is inert in $K$}\\
(1-\phi({\mathfrak P})p^{-s})(1-\phi(\bar{\mathfrak{P}})p^{-s})&\text{if $(p)=\mathfrak{P}\bar{\mathfrak{P}}$ in $K$.}
\end{cases}
\]
Therefore, $a_p=0$ if $p$ is inert in $K$. If $p$ splits into $\mathfrak{P}\bar{\mathfrak{P}}$, $a_p=\phi({\mathfrak P})+\phi(\bar{\mathfrak{P}})$. It is known that $\phi({\mathfrak P})+\phi(\bar{\mathfrak{P}})$ is a $p$-adic unit, hence $f$ is ordinary at $p$. Therefore, for a good prime $p\nmid N$, $a_p=0$ if and only if $f$ is supersingular at $p$. We fix such a $p$ which is odd.

Let $\OO$ be the ring of integers of $K$. We denote the conductor of $\phi$ by $\ff$. For an ideal $\mathfrak{a}$ of $K$, $K(\mathfrak{a})$ denotes the ray class field of $K$ of conductor $\mathfrak{a}$. We write $\K$ for the union $\cup_n K(p^n\ff)$. Then, the action of $G_\QQ$ on $V_f$ factors through $\Gal(\K/\QQ)$. The same is then true for $V_f(j)$ for all $j$ as $k_\infty\subset\K$.

More specifically, $V_f\cong V(\phi)\oplus\tau V(\phi)$ where $V(\phi)$ is the one-dimensional $E$-representation of $G_K$ associated to $\phi$ and $\tau$ is the complex conjugation. The action of $G_\QQ$ is given by
\[
\sigma(x,y)=\begin{cases}
(\sigma(x),\tau(\tau\sigma\tau)(y))&\text{if $\sigma\in G_K$},\\
((\tau\sigma\tau)(y),\tau\sigma(x))&\text{otherwise}.
\end{cases}
\]

In addition to assumptions (1) and (2), we assume for simplicity that the following holds:
\begin{itemize}
\item {\bf Assumption (3)}: $f$ is defined over $\QQ$ (i.e. $a_n\in\ZZ$ for all $n$) and $K$ has class number 1.
\end{itemize}

This is essential for the properties of elliptic units which we need to hold. Note that as a vector space, $V_f$ is isomorphic to $K_p$ (where $K_p$ denotes the completion of $K$ at $p$) and we can take $T_f$ to be the lattice corresponding to $\OO_p$. We write $\rho$ for the character given by
\[
\rho:G_K\rightarrow\textrm{Aut}(V_f/T_f(1))\cong \OO_p^\times.
\]
For simplicity, we write $A$ for $V_f/T_f(1)$ from now on.

Recall that $K_c$ denote the $\Zp$-cyclotomic extension of $K$. We write $K_m$ for the unique $\Zp^2$-extension of $K$ and $\mathfrak{L}$ denotes $\calO_p[[\Gal(K_m/K)]]$. Given a $\Zp[[\Gal(\K/K)]]$-module $Y$, we write $Y_F$ for $Y{\otimes}_{\Zp[[\Gal(\K/K)]]}\Zp[[\Gal(F/K)]]$ and $Y_F^\rho=Y_F(\rho^{-1})$
 where $F=K_c$ or $K_m$.

Let $F$ be an extension of $\QQ$. Following \cite{rubin85}, we define a modified Selmer group:
\[
\Sel_p'(f/F)=\ker\left(H^1(F,A)\rightarrow\prod_{v\nmid p} \frac{H^1(F_v,A)}{H^1_f(F_v,A)}\right).
\]

For a finite abelian extension $F$ of $K$, we define groups $C_F$, $E_F$ and $U_F$ as in \cite{pollackrubin04}: $U_F$ is the pro-$p$ part of the local unit group $(\OO_F\otimes\ZZ_p)^\times$, $E_F$ is the closure of the projection of the global units $\OO_F^\times$ into $U_F$ and $C_F$ is the closure of the projection of the subgroup of elliptic units (as defined in \cite[Section~1]{rubin91}, see also Section~\ref{EU} below) into $U_F$. We then define
\[
\C=\lim_\leftarrow C_{F},\ \mathcal{E}=\lim_\leftarrow E_{F}\ \ \textrm{and}\ \ 
\U=\lim_\leftarrow U_{F}
\]
where the inverse limits are taken over finite extensions $F$ of $K$ inside $\K$ and the connecting map is the norm map.

Finally, let $M$ be the maximal abelian $p$-extension of $\K$ which is unramified outside $p$ and write $\X$ for the Galois group of $M$ over $\K$.


\renewcommand{\a}{\mathfrak{a}}
\renewcommand{\b}{\mathfrak{b}}

\subsubsection{Elliptic units}\label{EU}

We now briefly review the definition of elliptic units associated to $K$. Let $\a$ and $\b$ be non-zero ideals of $\OO_K$ such that $\a$ is prime to $6\b$ and the natural map $\OO_K^\times\rightarrow(\OO_K/\b)^\times$ is injective. There exists an elliptic function on $\CC/\b$ with zeros and poles given by 0 (with multiplicity $N(\a)$) and the $\a$-division points respectively. There exists a unique such function if we impose some norm compatibility condition on its values as $\a$ varies. We write $_\a\theta_\b$ for this unique function and let $_\a z_\b=_\a\theta_\b(1)^{-1}$. Then, $_\a z_\b\in K(\b)^\times$ for any $\a$ and $\b$ as above. For a fixed $\b$, the group of elliptic units in $K(\b)$ is defined to be the group generated by $_\a z_\b^{\sigma}$  where $\sigma\in\Gal(K(\b)/K)$ and the roots of unity in $K(\b)$.


\subsection{Properties of $\Sel_p'$}

In this section, we generalise \cite[Theorem 2.1]{pollackrubin04}. We do this by generalising three results of \cite{rubin85}.

\begin{lemma}\label{selpi}
There is an isomorphism $\Sel_p'(f/K_c)\cong\Sel_p(f/K_c)$.
\end{lemma}

\begin{proof}  
By definitions, we have the following exact sequence:
\[
0\rightarrow\Sel_p(f/K_c)\rightarrow\Sel_p'(f/K_c)\rightarrow\frac{H^1(K_{c,p},A)}{H^1_f(K_{c,p},A)}.
\]
Therefore, it suffices to show that $H^1(K_{c,p},A)=H^1_f(K_{c,p},A)$. By \cite[Proposition 3.8]{blochkato},
\[
\left(\frac{H^1(K_{c,p},A)}{H^1_f(K_{c,p},A)}\right)^\vee=\lim_{\leftarrow}H^1_f(K_{p}^{(n)},T_{\bar{f}}(k-1)).
\]
 Hence, it suffices to show that the said inverse limit is $0$.

Note that $\Gal\left(K_{p,n}/K_{p}^{(n-1)}\right)\cong\Delta$, we have the inflation-restriction exact sequence
\[\begin{split}
0\rightarrow H^1(\Delta,T_{\bar{f}}(k-1)^{G_{K_{p,n}}})\rightarrow H^1(K_{p}^{(n-1)},T_{\bar{f}}(k-1))\rightarrow H^1(K_{p,n},T_{\bar{f}}(k-1))^\Delta\\\rightarrow H^2(\Delta,T_{\bar{f}}(k-1)^{G_{K_{p,n}}}).
\end{split}\]

As $K_p/\QQ_p$ is unramified, the proof of Lemma~\ref{inv} implies $T_{\bar{f}}(k-1)^{G_{K_{p,n}}}=0$ for all $n$. Therefore, 
\[
H^1(K_{p}^{(n-1)},T_{\bar{f}}(k-1))\cong H^1(K_{p,n},T_{\bar{f}}(k-1))^\Delta.
\]
By \cite[Theorem 0.6]{perrinriou00b}, we have $\displaystyle\lim_{\leftarrow}H^1_f(K_{n,p},T_{\bar{f}}(k-1))=0$, hence we are done.
\end{proof}

This corresponds to \cite[Theorem 2.1]{rubin85}, which holds for any infinite extensions of $K$ contained in $\K$. Since we have used a result on the inverse limit of $H^1_f$ over $K_{p,n}$, the proof above would unfortunately not work in such generality. 

We now generalise \cite[Proposition 1.1]{rubin85}.

\begin{lemma}\label{selhom}
There is an isomorphism $\Sel_p'(f/\K)\cong\Hom(\X,A)$.
\end{lemma}

\begin{proof}
Since the action of $G_K$ on $A$ factors through $\Gal(\K/K)$, we have $H^1(\K,A)\cong\Hom(G_\K,A)$. We can therefore identify $\Sel_p'(f/\K)$ with a subgroup of $\Hom(G_\K,A)$. Also, the triviality of the action implies that $A$ is unramified at all places of $\K$. Therefore, $H^1_f(\K_v,A)=H^1_\textrm{ur}(\K_v,A)$ for all $v\nmid p$ by \cite[Lemma 3.5(iv)]{rubin00}. Hence, $\Sel_p'(f/\K)$ corresponds to the subgroup $\Hom(\X,A)\subset\Hom(G_\K,A)$.
\end{proof}

Before we continue, we state a result of Rubin.

\begin{lemma}\label{Rubin0}
For $i=1,2$, $H^i(\K/K_c,A)=0$. 
\end{lemma}

\begin{proof}
See \cite[proof of Proposition 1.2]{rubin85}.
\end{proof}

This allows us to generalise \cite[Proposition 1.2]{rubin85}.

\begin{lemma}\label{seliso}
There is an isomorphism $\Sel_p'(f/K_c)\cong\Sel_p'(f/\K)^{{\rm Gal}(\K/K_c)}$.
\end{lemma}

\begin{proof}
We have the inflation-restriction exact sequence
\[
0\rightarrow H^1(\K/K_c,A)\rightarrow H^1(K_c,A)\stackrel{r}{\rightarrow}H^1(\K,A)^{{\rm Gal}(\K/K_c)}\rightarrow H^2(\K/K_c,A)
\]
where $r$ is the restriction map. Consider the following commutative diagram:
\[
\xymatrix{
H^1(K_c,A)\ar[r]^r\ar[d]&H^1(\K,A)\ar[d]\\
H^1(K_{c,v},A)/H^1_f(K_{c,v},A)\ar[r]&H^1(\K_{v'},A)/H^1_f(\K_{v'},A)
}
\]
where $v\nmid p$ is a place of $K_c$ and $v'$ is a place of $\K$ above $v$. It clearly implies that 
\[r\left(\Sel_p'(f/K_c)\right)\subset\Sel_p'(f/\K).\]

Write $v'$ for the place of $K_c(\ff)$ below $v'$, then $v'$ is unramified in $\K/K_c(\ff)$. Therefore, the map
\[
r_{v'}:H^1(I_{K_c(\ff)_{v'}},A)\rightarrow H^1(I_{\K_{v'}},A)
\]
where $I$ denotes the inertia group is injective. This implies that
\[
H^1(K_c(\ff)_{v'},A)/H^1_f(K_c(\ff)_{v'},A)\rightarrow H^1(\K_{v'},A)/H^1_f(\K_{v'},A)
\]
is injective because the $H^1_f$ coincide with $H^1_{\rm ur}$. But $\Gal(K_c(\ff)/K_c)$ has trivial Sylow $p$-subgroup, hence the bottom row of the commutative diagram above is injective. Therefore, we have 
\[
r^{-1}(\Sel_p'(f/\K))\subset\Sel_p'(f/K_c).\]
Hence, we have an exact sequence:
\[
0\rightarrow H^1(\K/K_c,A)\rightarrow\Sel_p'(f/K_c)\stackrel{r}{\rightarrow}\Sel_p'(f/\K)^{{\Gal}(\K/K_c)}\rightarrow H^2(\K/K_c,A).
\]
Hence, we are done by Lemma~\ref{Rubin0}.
\end{proof}

We can now give a generalisation of \cite[Theorem 2.1]{pollackrubin04}:

\begin{corollary}
$\Sel_p(f/K_c)\cong\Hom_{\OO}(\X^{\rho}_{K_c},K_p/\OO_p)$.
\end{corollary}

\begin{proof}
On combining Lemmas~\ref{selpi}, \ref{selhom} and \ref{seliso}, we have 
\begin{eqnarray*}
\Sel_p(f/K_c)&\cong&\Sel_p'(f/K_c)\\
&\cong&\Sel_p'(f/\K)^{\textrm{Gal}(\K/K_c)}\\
&\cong&\Hom(\X,A)^{\textrm{Gal}(\K/K_c)}
\end{eqnarray*}
But $A|_{G_K}\cong K_p/\OO_p(\rho)$, hence the result.
\end{proof}


\subsection{Reciprocity law}
In this section, we generalise the reciprocity law given by \cite[Theorem 5.1]{pollackrubin04}. We first review a result of Rubin.

\begin{theorem}\label{rubin}
The $\mathfrak{L}$-module $\C^\rho_{K_m}$ is free of rank $1$. 
\end{theorem}
\begin{proof}
It follows from \cite[Theorem 7.7]{rubin91}.\end{proof}

We now generalise \cite[Proposition 4.1]{pollackrubin04}:
\begin{lemma}\label{U}
$H^1_f(K_{c,p},A)\cong\Hom_\OO(\U_{K_c}^\rho,K_p/\OO_p)$.
\end{lemma}

\begin{proof}
As in the proof of Lemma~\ref{selhom}, we have $H^1(\K_p,A)\cong\Hom(G_{\K_p},A)$. But we also have an isomorphism $H^1(K_{c,p},A)\cong H^1(\K_{p},A)^{\textrm{Gal}(\K_p/K_{c,p})}$ by the inflation-restriction sequence and Lemma~\ref{Rubin0}. Hence, by local class field theory, we have
\begin{eqnarray*}
H^1(K_{c,p},A)&\cong&\Hom(G_{\K_p},A)^{\textrm{Gal}(\K_p/K_{c,p})}\\
&\cong&\Hom_{\OO_p}(\U,A)
\end{eqnarray*}
(see \cite[Proposition 5.2]{rubin87}). By the proof of Lemma \ref{selpi}, we have $H^1_f(K_{c,p},A)\cong H^1(K_{c,p},A)$, hence we are done.\end{proof}

In particular, we have a pairing $<,>:H^1_f(K_{c,p},A)\times\U_{K_c}^\rho\rightarrow K_p/\OO_p$. We now prove the explicit reciprocity law.

\begin{proposition}\label{law}
There exists a generator $\xi$ of $\C^\rho_{K_m}$ over $\mathfrak{L}$ such that for any finite extension $F$ of $K$ contained in $K_c$, $\theta$ a character on $G=\Gal(F/K)$, $x\in H^1_f(F_p,A)$ and $r$ a non-negative integer, we have
\begin{equation}\label{rec}
\sum_{\sigma\in G}\theta(\sigma)<x^\sigma\otimes p^{-r},\xi>=p^{-r}\frac{L(f_{\theta^{-1}},1)}{\Omega_f^\pm}\left[\sum_{\sigma\in G}\theta(\sigma)\exp^{-1}_{F_p,V_f(1)}(x^\sigma),\bar{\omega}_{-1}\right]
\end{equation}
where $\theta(-1)=\pm$ and $\exp^{-1}_{F_p,V_f(1)}$ is the inverse of the exponential map
\[
\exp_{F_p,V_f(1)}:F_p\otimes \DD(V_f(1))/\DD^0(V_f(1))\stackrel{\sim}{\rightarrow}H^1_f(F_p,V_f(1)).
\]
\end{proposition}

\begin{proof}
Let $z_{p^\infty\ff}=(z_{p^n\ff})_n$ be the system of norm-compatible elliptic units in
$\displaystyle
\lim_{\leftarrow}K(p^n\ff)
$
defined in \cite[Section 16.5]{kato04}, then $_\a z_{p^n\ff}$ is a multiple of $z_{p^n\ff}$ for all $\a$ and $p^n\ff$ satisfying the conditions in Section~\ref{EU}. Therefore, if we write $\xi$ as its image in $\C^\rho_{K_m}$, it must be a generator of $\C^\rho_{K_m}$ over $\mathfrak{L}$ by Theorem~\ref{rubin}.

Let $x\in H^1_f(F_p,T_f(1))$ and $y\in H^1(F_p,T_{\bar{f}}(k-1))$, we have
\begin{eqnarray*}
\sum_{\sigma\in G}\theta(\sigma)[x^\sigma,y]&=&\sum_{\sigma\in G}\theta(\sigma){\rm Tr}_{F/K}\left[\exp^{-1}_{F_p,V_f(1)}(x^\sigma),\exp^*_{F_p,V_{\bar{f}(k-1)}}(y)\right]\\
&=&\sum_{\sigma,\tau\in G}\theta(\sigma)\left[\exp^{-1}_{F_p,V_f(1)}(x^{\sigma\tau}),\exp^*_{F_p,V_{\bar{f}(k-1)}}(y^\tau)\right]\\
&=&\sum_{\sigma,\tau\in G}\theta(\sigma\tau)\theta^{-1}(\tau)\left[\exp^{-1}_{F_p,V_f(1)}(x^{\sigma\tau}),\exp^*_{F_p,V_{\bar{f}(k-1)}}(y^\tau)\right]\\
&=&\left[\sum_{\sigma\in G}\theta(\sigma)\exp^{-1}_{F_p,V_f(1)}(x^{\sigma}),\sum_{\tau\in G}\theta^{-1}(\tau)\exp^*_{F_p,V_{\bar{f}(k-1)}}(y^\tau)\right].
\end{eqnarray*}
Consider the Kummer exact sequences:
\[
\xymatrix{
\C\ar[r]\ar[d]&\U\ar[d]\\
\displaystyle\lim_{\leftarrow}H^1(\OO_{K'}[1/p],\OO_p(1))\ar[r]\ar[d]^{\otimes\rho\chi^{k-2}}&\displaystyle\lim_{\leftarrow}H^1(K'_p,\OO_p(1))\ar[d]^{\otimes\rho\chi^{k-2}}\\
\displaystyle\lim_{\leftarrow}H^1(\OO_{K'}[1/p],T_{\bar{f}}(k-1))\ar[r]&\displaystyle\lim_{\leftarrow}H^1(K'_p,T_{\bar{f}}(k-1)).
}
\]
By \cite[Proposition~15.9 and (15.16.1)]{kato04}, the image of $z_{p^\infty\ff}$ in $\displaystyle\lim_{\leftarrow}H^1(\OO_{K'}[1/p],T_{\bar{f}}(k-1))$ is $\kato$ (up to a twist) and so $\xi$ satisfies
\[
\sum_{\tau\in G}\theta^{-1}(\tau)\exp^*_{F_p,V_{\bar{f}(k-1)}}(\xi^\tau)=\displaystyle\frac{L(f_{\theta^{-1}},1)\bar{\omega}_{-1}}{\Omega_f^\pm}.
\]
Therefore, we have:
\[
\sum_{\sigma\in G}\theta(\sigma)<x^\sigma\otimes p^{-r},\xi>=p^{-r}\left[\sum_{\sigma\in G}\theta(\sigma)\exp^{-1}_{F,V_f(1)}(x^{\sigma}),\frac{L(f_{\theta^{-1}},1)\bar{\omega}_{-1}}{\Omega_f^\pm}\right]
\]
as required.
\end{proof}


\subsection{Proof of the main conjecture}

On replacing $\Qpn$ by $K_{p,n}$, we define $H^1_f(K_{p,n},W)^\pm$ and hence $\Sel_p^\pm(f/K_\infty)$ as in Section~\ref{selmer} where $W=A$ or $T_f(1)$. Let $\mathcal{G}=\Gal(K/\QQ)$. As in the proof of Lemma~\ref{selpi}, the inflation-restriction exact sequence implies that $H^1(\Qpn,W)\cong H^1(K_{p,n},W)^\mathcal{G}$ for $W=A$ or $T_f(1)$, so we recover $\Sel_p^{\pm}(f/k_\infty)$ on taking $\mathcal{G}$-invariant. Similarly, on replacing $\Qpn$ and $K_{p,n}$ by $\QQ_p^{(n-1)}$ and $K_p^{(n-1)}$ respectively, we define the $\pm$-Selmer groups $\Sel_p^\pm(f/\QQ_c)$ and $\Sel_p^\pm(f/K_c)$. Under our assumptions, they coincide with the $\Delta$-invariants of $\Sel_p^\pm(f/k_\infty)$ and $\Sel_p^\pm(f/K_\infty)$ respectively. Analogously, we have $H^1_\pm(F,T_{\bar{f}}(k-1))$ for $F=K_{p,n}$, $K_p^{(n-1)}$ or $\Qp^{(n-1)}$. Since $K_p/\Qp$ is unramified, all the results from the previous sections generalise directly on replacing $\Qp$ by $K$.

Via the isomorphism defined in Lemma \ref{U}, we define $\V^\pm\subset\U^\rho_{K_c}$ to be the subgroup corresponding to the elements of $\Hom_{\OO}\left(H^1_f(K_{c,p},A),K_p/\OO_p\right)$ which factor through $H^1_f(K_{c,p},A)^\pm$. Then, by \cite[Theorem~4.3]{pollackrubin04}, $\Sel_p^\pm(f/K_c)\cong\Hom_\OO\left(\X_{K_c}^\rho/\alpha(\V^\pm),K_p/\OO_p\right)$ where $\alpha$ is the Artin map on $\U$, which enables us to generalise \cite[Theorem~7.2]{pollackrubin04}:

\begin{theorem}
Let $s^\pm$ be as given by Corollary~\ref{inteim}, then
\[
\Char_{\Lambda_{\OO_p}(\Gamma)}\left(\Hom_\OO\left(\Sel_p^\pm(f/K_c),K_p/\OO_p\right)\right)=\left(p^{-s^\pm}L_p^\pm\right).
\]
\end{theorem}

\begin{proof}
By the above isomorphism and \cite[Theorem 6.3]{pollackrubin04}, we have:
\begin{eqnarray*}
&&\Char_{\Lambda_{\OO_p}(\Gamma)}\left(\Hom_\OO\left(\Sel_p^\pm(f/K_c),K_p/\OO_p\right)\right)\\
&=&\Char_{\Lambda_{\OO_p}(\Gamma)}\left(\X_{K_c}^\rho/\alpha(\V^\pm)\right)\\
&=&\Char_{\Lambda_{\OO_p}(\Gamma)}\left(\U^\rho_{K_c}/(\V^\pm+\C^\rho_{K_c})\right).
\end{eqnarray*}

By Corollary \ref{trivialchar}, the quotient $H^1(\QQ_{c,p},T_{\bar{f}}(k-1))/H^1_\pm(\QQ_{c,p},T_{\bar{f}}(k-1))$ is free of rank one over $\Lambda(\Gamma)$. Hence, by \eqref{defn} and the proofs of Lemma~\ref{inte} and Corollary~\ref{trivialchar}, the $\Lambda(\Gamma)$-module $\Hom\left(H^1_f(\QQ_{c,p},T_f(1))^\pm,\ZZ_p\right)$ is also free of rank one and it has a generator $f_\pm$ such that
\begin{equation}
\sum_{\sigma\in G_n}f_\pm(\exp_{n,1}(\gamma_{n,1}(\eta^\pm_1)^\sigma))\sigma\equiv p^{s^\pm}\log_{p,k}^\pm\mod(\gamma^{p^{n-1}}-1)\label{inter}
\end{equation}
Note that we have abused notation by writing $\exp_{n,1}(\gamma_{n,1}(\eta^\pm_1))$ for its image in $H^1(\Qp^{(n-1)},T_f(1))$ under the corestriction.

As in \cite[Theorems 7.1 and 7.2]{pollackrubin04}, we have
\begin{eqnarray*}
\Hom\left(H^1_f(\QQ_{c,p},A)^\pm,\QQ_p/\ZZ_p\right)&\cong&\Hom\left(H^1_f(\QQ_{c,p},T_f(1))^\pm,\ZZ_p\right),\\
\Hom_\OO\left(H^1_f(K_{c,p},A)^\pm,K_p/\OO_p\right)&\cong&\Hom\left(H^1_f(\QQ_{c,p},A)^\pm,\QQ_p/\ZZ_p\right)\otimes\OO_p.
\end{eqnarray*} 

Let $\mu^\pm$ and $\vartheta^\pm$ be the images of $f_\pm$ and $\xi$ from Proposition~\ref{law} in $\Hom_\OO\left(H^1_f(K_{c,p},A)^\pm,K_p/\OO_p\right)$ respectively. Then $\vartheta^\pm=h^\pm\mu^\pm$ for some $h^\pm\in\Lambda_{\OO_p}(\Gamma)$. As in \cite[proof of Theorem 7.2]{pollackrubin04}, there is an isomorphism $\U^\rho_{K_c}/(\V^\pm+\C^\rho_{K_c})\cong\Lambda_{\OO_p}(\Gamma)/h^\pm\Lambda_{\OO_p}(\Gamma)$. Hence we have:
\[
\Char_{\Lambda_{\OO_p}(\Gamma)}\left(\Hom_\OO\left(\Sel_p^\pm(f/K_c),K_p/\OO_p\right)\right)=h^\pm\Lambda_{\OO_p}(\Gamma).
\]

Let $F$ be a finite extension of $K$ contained in $K_c$, $\theta$ a character of $G$, the Galois group of $F$ over $K$, $x\in H^1_f(F_p,A)$, $r$ and integer, then $\vartheta^\pm=h^\pm\mu^\pm$ implies
\begin{equation}
\sum_{\sigma\in G}\theta(\sigma)\vartheta^\pm(x^\sigma\otimes p^{-r})=\theta(h^\pm)\sum_{\sigma\in G}\theta(\sigma)\mu^\pm(x^\sigma\otimes p^{-r})\label{compare}
\end{equation}

We now take $x=\exp_{n,1}(\gamma_{n,1}(\eta^\pm_1))$. By (\ref{inter}), the RHS of (\ref{compare}) is just $p^{-r+s^\pm}\theta(h^\pm)\theta(\log_{p,k}^\pm)$.  Then, \eqref{rec} implies that the LHS of (\ref{compare}) equals to the following:
\[
p^{-r}\frac{L(f_{\theta^{-1}},1)}{\Omega_f^\delta}\left[\sum_{\sigma\in G}\theta(\sigma)\gamma_{n,1}(\eta^\pm_1)^\sigma,\bar{\omega}_{-1}\right]
\]
 where $\delta=\theta(-1)$. We now compute $\sum_{\sigma\in G}\theta(\sigma)\gamma_{n,1}(\eta^\pm_1)^\sigma$.

Take $F$ to be $K_p^{(n-1)}$ and $\theta$ a character of conductor $p^{n}$. Then
\begin{eqnarray*}
\sum_{\sigma\in G}\theta(\sigma)\gamma_{n,1}(\eta^\pm)^\sigma&=&\sum_{\sigma\in G}\frac{\theta(\sigma)}{p^{n}}\left(\sum_{i=0}^{n-1}\zeta_{p^{n-i}}^\sigma\otimes\vp^{i-n}(\eta^\pm_1)+(1-\vp)^{-1}(\eta^\pm_1))\right)\\
&=&p^{-n}\sum_{\sigma\in G}\theta(\sigma)\zeta_{p^n}^\sigma\otimes\vp^{-n}(\eta_1^\pm)\\
&=&p^{-n}\tau(\theta)\vp^{-n}(\eta^\pm_1)
\end{eqnarray*} 
where $\tau(\theta)$ denotes the Gauss sum of $\theta$. Since $\vp^2+\epsilon(p)p^{k-3}=0$ on $\DD(V_f(1))$, we have
\begin{eqnarray*}
\vp^{-n}(\eta_1^-)&=&(-\epsilon(p)p^{k-3})^{\frac{-n-1}{2}}p^{-1}\vp(\omega)_1/[\vp(\omega),\bar{\omega}]\ \ (\text{for } n\ \textrm{odd}),\\
\vp^{-n}(\eta_1^+)&=&(-\epsilon(p)p^{k-3})^{\frac{-n}{2}}\vp(\omega)_1/[\vp(\omega),\bar{\omega}]\ \ (\text{for }n\ \textrm{even}).
\end{eqnarray*}
Hence, (\ref{compare}) implies:
\begin{eqnarray*}
p^{s^-}\theta(h^-)\theta(\log_{p,k}^-)&=&(-\epsilon(p)p^{k-1})^{\frac{-n-1}{2}}\tau(\theta)\frac{L(f_{\theta^{-1}},1)}{\Omega_f^\delta}\ \ (\text{for }n\ \textrm{odd}),\\
p^{s^+}\theta(h^+)\theta(\log_{p,k}^+)&=&(-\epsilon(p)p^{k-1})^{\frac{-n}{2}}\tau(\theta)\frac{L(f_{\theta^{-1}},1)}{\Omega_f^\delta}\ \ (\text{for }n\ \textrm{even}).
\end{eqnarray*}Therefore, by the interpolating properties of $L_p^\pm$ at these characters, we have:
\begin{eqnarray*}
p^{s^-}\theta(h^-)&=&\theta(L_p^-)\ \ (\text{for }n\ \textrm{odd}),\\
p^{s^+}\theta(h^+)&=&\theta(L_p^+)\ \ (\text{for }n\ \textrm{even}).
\end{eqnarray*}
But $h^\pm$ and $L_p^\pm$ are both $O(1)$ and the above holds for infinitely many $n$, so $h^\pm=p^{-s^\pm}L_p^\pm$. Hence we are done.
\end{proof}

By taking $\mathcal{G}$-invariants, we have the following.
\begin{corollary}
$\Char_{\Lambda(\Gamma)}\left(\Sel_p^\pm(f/\QQ_c)^\vee\right)=\left(p^{-s^\pm}L_p^\pm\right)$.
\end{corollary}

\bibliographystyle{amsalpha}  
\bibliography{references}
\end{document}